\documentclass[11pt,a4paper]{amsart}

\usepackage[english]{babel}
\usepackage[T1]{fontenc}
\usepackage[utf8]{inputenc}

\usepackage{geometry}
\usepackage{relsize}
\usepackage[nodisplayskipstretch]{setspace}

\usepackage[alphabetic]{amsrefs} 
\DefineSimpleKey{bib}{primaryclass}{}
\DefineSimpleKey{bib}{archiveprefix}{}
\newcommand*{\bracketize}[1]{[#1]}
\BibSpec{arxiv}{%
  +{}{\PrintAuthors}{author}
  +{,}{ \textit}{title}
  +{}{ \parenthesize}{date}
  +{,}{ arXiv:}{eprint}
  +{}{ \bracketize}{primaryclass}
}
\usepackage[dvipsnames]{xcolor}
\usepackage{graphicx}
\usepackage{tikz}
\usetikzlibrary{positioning,arrows.meta,cd,babel}
\usepackage[percent]{overpic}
\usepackage[shortlabels]{enumitem}
\setlist[itemize]{topsep=3pt,itemsep=3pt}

\graphicspath{ {figures/} }
\usepackage{amssymb}      
\usepackage{mathtools}    
\usepackage{mathrsfs}     
\usepackage{dsfont}       
\usepackage{xfrac}        
\usepackage{nccmath}

\definecolor{myRED}{HTML}{E6332A}
\definecolor{myGRAY}{HTML}{8E8E8E}
\definecolor{myMIDDGRAY}{HTML}{E5E5E5}
\definecolor{myMIDGRAY}{HTML}{E2E2E2}
\definecolor{myDARKGRAY}{HTML}{404647}
\definecolor{myLIGHTGRAY}{HTML}{F9F9F9}
\definecolor{myBLUE}{HTML}{03468F}
\definecolor{myPURPLE}{HTML}{8C54D0}
\definecolor{myGREEN}{HTML}{007355}
\definecolor{myYELLOW}{HTML}{FFD300}
\definecolor{myDARKYELLOW}{HTML}{FF9000}

\usepackage{hyperref}
\hypersetup{
  colorlinks=true,
  linkcolor=blue,
  urlcolor=blue,
  citecolor=blue
}

\numberwithin{equation}{section}

\newtheorem{theorem}{Theorem}[section]
\newtheorem{lemma}[theorem]{Lemma}
\newtheorem{proposition}[theorem]{Proposition}
\newtheorem{corollary}[theorem]{Corollary}
\newtheorem{claim}[theorem]{Claim}

\newtheorem*{theorem*}{Theorem}
\newtheorem*{pushinglemma*}{Pushing Lemma}
\newtheorem*{lemmasep*}{Proposition 7.6}
\newtheorem*{definition*}{Definition}

\usepackage{alphalph}
\makeatletter
\newtheorem{thmx}{Theorem}

\makeatother



\theoremstyle{definition}
\newtheorem{definition}[theorem]{Definition}

\newtheorem{remark}[theorem]{Remark}

\newcommand{\mycomment}[1]{}

\newcommand{\sspc}{\hspace*{0.04cm}}
\newcommand{\spc}{\hspace*{0.08cm}}
\newcommand{\pp}{{\sspc\prime}}

\newcommand{\homeo}{\mathrm{Homeo}}

\newcommand{\mcg}{\mathrm{MCG}}

\newcommand{\ind}{\mathrm{Ind}}
\newcommand{\str}{\mathrm{Str}}

\newcommand{\Tp}{\dot\theta}

\renewcommand{\O}{O}
\renewcommand{\P}{\mathcal P}
\renewcommand{\H}{\mathcal H}
\newcommand{\OO}{\mathcal O}

\newcommand{\R}{\mathbb R}

\newcommand{\Z}{\mathbb Z}

\newcommand{\F}{\mathcal F}

\newcommand{\W}{\mathcal W}

\newcommand{\ZZ}{\mathbb Z/4\mathbb Z}

\newcommand{\class}[1]{\big[\hspace*{0.08cm} #1 \hspace*{0.08cm} \big]}

\newcommand{\orb}{\mathrm{Orb}(f)}

\title[An index theory for transverse trajectories]{An index theory for transverse trajectories}

\author[N. Schuback]{Nelson Schuback}
\address{Sorbonne Université and Université Paris Cité, CNRS, IMJ-PRG, F-75005 Paris, France}
\email{\href{mailto:nelson.schuback@imj-prg.fr}{nelson.schuback@imj-prg.fr}}
\thanks{This project has received funding from the European Union’s Horizon 2020 research and innovation programme under the Marie Skłodowska-Curie grant agreement No 945332. \includegraphics*[scale = 0.03]{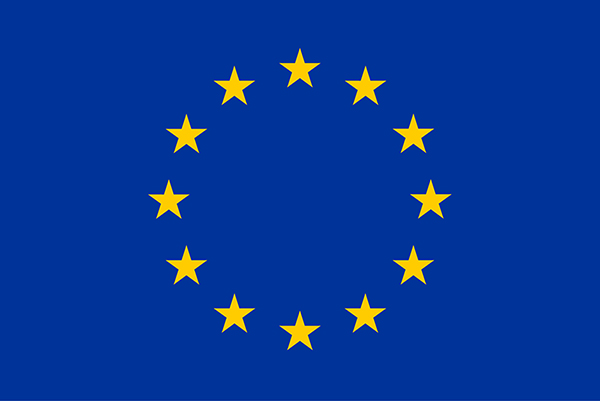}}

\begin{document}

\begin{abstract}
In this work, we present an alternative definition of the Le Roux index \cite{LEROUX_2017}, which generalizes the Poincaré–Hopf index for non-singular planar flows to the broader setting of Brouwer homeomorphisms. This new approach answers a question raised by Le Roux by establishing a connection between the index of a Brouwer homeomorphism and the structure of its transverse foliations, in the sense of Le Calvez \cite{lec1}.
\end{abstract}

\vspace*{0.5cm}

\maketitle


\tableofcontents

\vspace*{-0.6cm}

\setlength{\baselineskip}{1.25\baselineskip}

\section{Introduction}

Let $X$ be a (smooth) non-singular vector field on $\R^2$. There exists a unique foliation of the plane $\F$ such that the leaves of $\F$ are the integral curves of $X$. Meaning, each leaf of $\F$ is a topological line corresponding to a trajectory of the flow $\{\Phi_t\}_{t \in \R}$ given as solution of 
\vspace{0.2cm}
$$ \begin{cases}
    \ \sspc \dfrac{d}{dt}\Phi_t(x) = X(x), \\[2ex]
    \ \sspc \Phi_0 (x) = x.
\end{cases}
$$

\vspace{0.1cm}
The \textit{Poincaré-Hopf index} between two leaves $\phi_1$ and $\phi_2$ of $\F$, denoted by $\ind(\F,\phi_1,\phi_2)$, can be defined as follows. We remark that, for any two distinct leaves $\phi_1,\phi_2 \in \F$, there exists an
orientation-preserving diffeomorphism $h:\R^2\longrightarrow \R^2$ satisfying $h(\phi_i) = \R_{\sspc i}\sspc$ for each $i\in\{1,2\}$, where $\R_i$ denotes the horizontal $\R\times\{i\}$. Observe that the pushforward vector field $h_*X$ is horizontal over the lines $h(\phi_1)$ and $h(\phi_2)$. This means that, for any path $\alpha:[0,1]\longrightarrow \R$ joining $h(\phi_1)$ to $h(\phi_2)$, the winding number of $h_*X$ along $\alpha$ is an integer multiple of $1/2$. Recall that\pagebreak \  this winding number is given by the expression
$$ \text{Wnd}(h_*X,\alpha) = \frac{1}{2\pi}\bigl(\sspc\widetilde{s}\spc(1) - \widetilde{s}\spc(0) \bigr),$$
where $\widetilde{s\sspc}:[0,1]\longrightarrow\R$ is a $\text{mod}_{2\pi}$ lift of the function $s:[0,1]\longrightarrow \R/2\pi\Z$ that assigns to each parameter $t\in[0,1]$ the angle of the vector $h_*X(\alpha(t))$ with respect to the positive horizontal direction. Via two different connectedness arguments, one can show that $\text{Wnd}(h_*X,\alpha)$ is independent of the choice of $h$ and $\alpha$. In conclusion, this winding number defines a notion of index between the leaves $\phi_1$ and $\phi_2$ of the foliation $\F$, which is denoted by $\ind(\F,\phi_1,\phi_2)$. This index can be unimbigously associated to the time-one map $\Phi_1$ and two of its orbits $\O_{\sspc 1}=\O(\Phi_1,x_1)$ and $\O_{\sspc 2}=\O(\Phi_1,x_2)$, for any $x_1 \in \phi_1$ and $x_2 \in \phi_2$, thus defining $\ind(\Phi_1,\O_{\sspc 1},\O_{\sspc 2}) := \ind(\F,\phi_1,\phi_2)$.

The notion of Brouwer homeomorphisms, i.e., fixed-point–free, orientation-preserving homeomorphisms of the plane, strictly generalizes that of time-one maps of non-singular flows \cite{Andrea65}. Surprisingly, many dynamical properties of non-singular flows extend to the general context of Brouwer homeomorphisms, despite the lack of a flow structure. For instance, Brouwer's Plane Translation Theorem \cite{Brouwer} implies that every orbit of a Brouwer homeomorphism is a proper embedding of the integers into the plane, and thus generalizes the Poincaré-Bendixson Theorem.
In this context, a natural question arises:
\begin{center}
    \textit{Is there a way to generalize the Poincaré-Hopf index to Brouwer homeomorphisms?}
\end{center}

This question was first addressed by Le Roux in \cite{LEROUX_2017}. There, building on Handel's Homotopy Brouwer Theory \cite{handel99}, Le Roux introduces a notion of index for Brouwer homeomorphisms with respect to two orbits, which we hereby refer to as the \textit{Le Roux index}. 

Let $f:\R^2\longrightarrow \R^2$ be a Brouwer homeomorphism, and let $\O_{\sspc 1}$ and $\O_{\sspc 2}$ be two orbits of $f$.\break The Brouwer mapping class of $f$ relative to $\O_{\sspc 1}$ and $\O_{\sspc 2}$, denoted by $\big[\spc f,\sspc\{\O_{\sspc 1},\O_{\sspc 2}\}\spc\big]$, is the isotopy class of $f$ relative $\O_{\sspc 1}\sspc \cup\sspc \O_{\sspc 2}$. Based on Handel's results in \cite{handel99}, we know that there exists an orientation-preserving homeomorphism $h:\R^2\longrightarrow \R^2$ that satisfies $h(\O_{\sspc i}) = \Z \times \{i\}$,\break for each $i\in\{1,2\}$, and that conjugates $\big[\spc f,\sspc\{\O_{\sspc 1},\O_{\sspc 2}\}\spc\big]$ to one of the following classes:
$$ \bigl[\spc T\sspc,\spc\Z_{1,2}\sspc\bigr]\sspc, \quad \class{T^{-1},\spc\Z_{1,2}}\sspc, \quad \bigl[\spc R\sspc,\spc\Z_{1,2}\sspc\bigr]\sspc, \quad \class{R^{-1},\spc\Z_{1,2}}\sspc,$$
where $\Z_{1,2}$ is the set $\Z \times \{1,2\}$, the map $T$ is the unitary horizontal translation on the plane, and $R$ is the time-one map of the standard Reeb flow on the plane (see Section \ref{sec:homotopy_brouwer_theory} for further details).
The set of all such homeomorphisms $h$ for $f$, $\O_1$ and $\O_2$ is denoted by $\textup{Handel}(\sspc f,\O_{\sspc 1},\O_{\sspc 2} )$.


Let $X_f$ be the $C^0$-vector field on the plane defined by the expression $X_f(x) = f(x)-x$.\break Observe that, for any map $h\in\text{Handel}(f,\O_{\sspc 1},\O_{\sspc 2})$, the vector field $X_{h\sspc\circ\sspc f \sspc\circ\sspc h^{-1}}$
is horizontal on $\Z_1$ and $\Z_2$. Thus, the winding number $\text{Wnd}(X_{h\sspc\circ\sspc f \sspc\circ\sspc h^{-1}},\alpha)$ along any path $\alpha$ joining $\Z_1$ to $\Z_2$ must be an integer multiple of $1/2$. In \cite{LEROUX_2017}, Le Roux shows that this winding number is independent of the choice of $h\in \text{Handel}(f,\O_{\sspc 1},\O_{\sspc 2})$ and $\alpha$, and therefore defines a notion of index for the Brouwer homeomorphism $f$ with respect to $\O_{\sspc 1}$ and $\O_{\sspc 2}$, denoted by $\ind(\sspc f,\O_{\sspc 1},\O_{\sspc 2} )$.

\vspace*{0.36cm}

\begin{figure}[h!]
    \hspace*{-1.2cm}\begin{overpic}[width=8.5cm, height=3.6cm, tics=10]{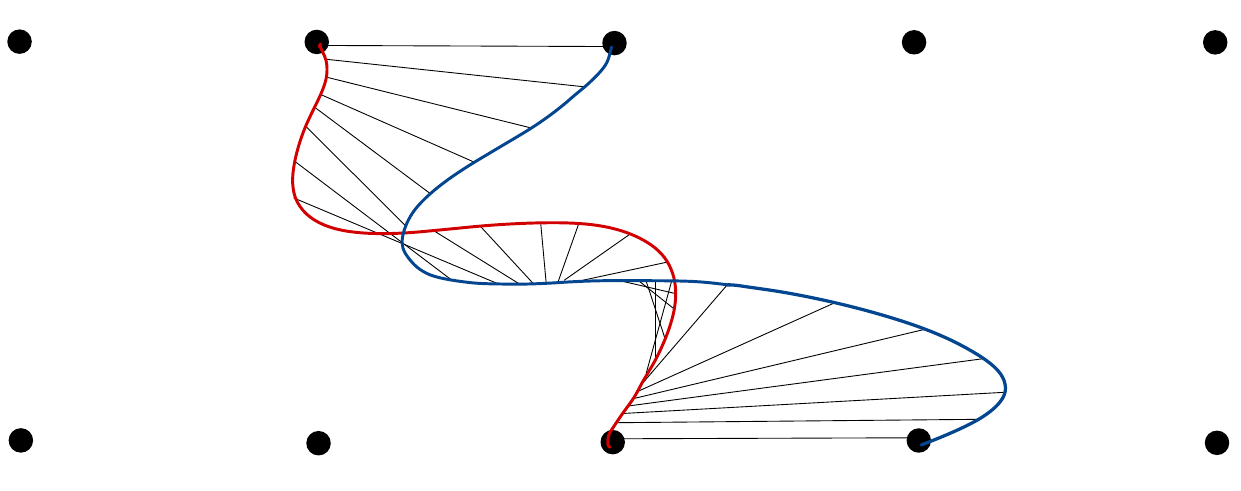}
        \put (28,43.8) {{\color{black}\large$\displaystyle X_{h\sspc\circ\sspc f \sspc\circ\sspc h^{-1}}$}}
        \put (105,37.2) {{\color{black}\large$\displaystyle \Z_2 = h(\O_2)$}}
        \put (105,1.8) {{\color{black}\large$\displaystyle \Z_1=h(\O_1)$}}
        \put (60,20) {{\color{myBLUE}\large$\displaystyle h\circ f \circ h^{-1} \circ \alpha$}}
        \put (18,26) {{\color{myRED}\large$\displaystyle \alpha$}}
\end{overpic}
\end{figure}

As expected, the Le Roux index generalizes the Poincaré–Hopf index for time-one maps of non-singular flows. In contrast with the flow case, however, it does not admit a natural interpretation in terms of a foliation index. Furthermore, since it depends on a specific conjugacy chosen according to criteria from Homotopy Brouwer Theory rather than on intrinsic properties of the Brouwer homeomorphism $f$, it becomes unclear how the Le Roux index might reflect the qualitative dynamical features of $f$. This issue is emphasized in \cite{LEROUX_2017}, where Le Roux conjectures that there may exist a more intrinsic and possibly more intuitive definition of the index \(\ind(f,\O_{\sspc 1},\O_{\sspc 2})\), involving a conjugacy criterion based on the qualitative dynamical features of $f$ with respect to the orbits \(\O_{\sspc 1}\) and \(\O_{\sspc 2}\).

In this work, we adress Le Roux's question by providing a foliation-based definition of the index for Brouwer homeomorphisms that coincides with the Le Roux index. First, we introduce an index \(\ind(\sspc\F,\sspc\Gamma_1,\Gamma_2)\) between two oriented topological lines $\Gamma_1$ and $\Gamma_2$ positively transverse to a foliation \(\F\) (see Proposition \ref{prop:index_invariance}). Below, we give an intuitive description of this index.

\vspace*{0.2cm}

Let $\F$ be an oriented foliation of $\R^2$ and let $\Gamma_1$ and $\Gamma_2$ be oriented topological lines positively transverse to $\F$. If $\Gamma_1$ and $\Gamma_2$ intersect, we will define $\ind(\F,\Gamma_1,\Gamma_2):=0$. Thus, we can assume that $\Gamma_1\cap\Gamma_2=\varnothing$. Consider an orientation-preserving homeomorphism $h:\R^2\longrightarrow \R^2$ satisfying $h(\Gamma_i) = \R_{\sspc i}$ for $i \in \{1,2\}$, and two paths $\alpha,\alpha':[0,1]\longrightarrow \R^2$ satisfying $\alpha(0) \in \Gamma_1$, $\alpha(1) \in \Gamma_2$, and 
$$\alpha'(t) \in \phi^+_{\alpha(t)}\setminus\{\alpha(t)\} \sspc, \quad \forall t\in[0,1],$$
\noindent
where $\phi_{\alpha(t)}\in \F$ denoted the leaf passing through the point $\alpha(t)$, and $\phi^+_{\alpha(t)}$ denotes the positive half of $\phi_{\alpha(t)}$ starting at the point $\alpha(t)$. Roughly speaking, the path $\alpha'$ is obtained by pushing the path $\alpha$ slightly along the positive direction of the leaves of $\F$.

Exclusively for this intuitive definition, we suppose that $h(\phi_{\alpha(0)})$ and $h(\phi_{\alpha(1)})$ are vertical near the points $h(\alpha(0))$ and $h(\alpha(1))$, and that $h(\alpha'(0))$ and $h(\alpha'(1))$ are located in these vertical parts. With this extra hypothesis, we observe that the winding number of the vector 
$$X_{h,\alpha,\alpha'} (t) = h\circ \alpha'(t) - h\circ \alpha(t),$$
as $t$ varies from $0$ to $1$, is an integer multiple of $1/2$. In Chapter \ref{chap:foliation_index}, we present a general theory, based on the notion of topological angle, that implicitly implies the independence of this winding number from the choice of $h$, $\alpha$ and $\alpha'$. This winding number defines the index \(\ind(\F,\Gamma_1,\Gamma_2)\).

\vspace*{0.3cm}
\begin{figure}[h!]
    \center\hspace*{-1cm}\begin{overpic}[width=7.3cm, height=4.2cm, tics=10]{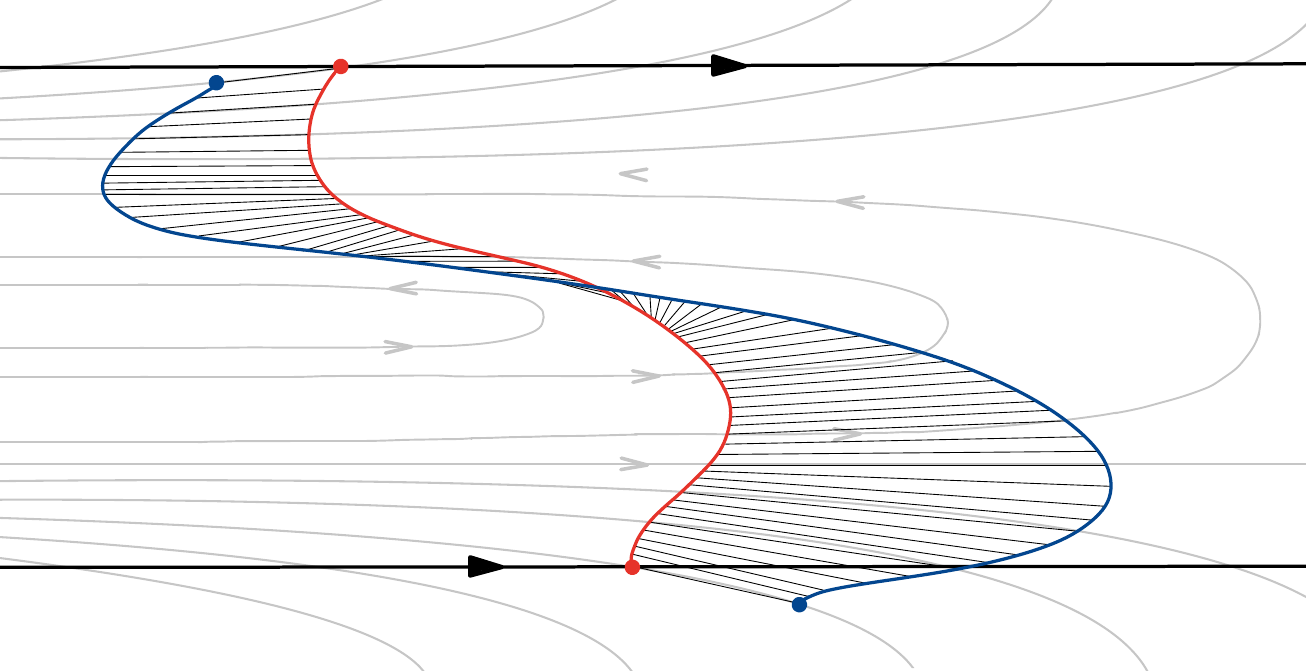}
        \put (16,56) {\colorbox{white}{\color{black}\large$\displaystyle \rule{0cm}{0.28cm}\quad \quad \ \ $}}
        \put (17.5,56.5) {{\color{black}\large$\displaystyle X_{h,\alpha,\alpha^\pp}$}}
        \put (105,50.5) {{\color{black}\large$\displaystyle \R_2 = h(\Gamma_2)$}}
        \put (105,7.5) {{\color{black}\large$\displaystyle \R_1=h(\Gamma_1)$}}
        \put (100,30) {{\color{myGRAY}\large$\displaystyle \F$}}
        \put (78.5,25.5) {{\color{myBLUE}\large$\displaystyle h\circ  \alpha^\pp$}}
        \put (40,21.5) {{\color{myRED}\large$\displaystyle h \circ \alpha$}}
\end{overpic}
\end{figure}

Finally, we prove that this foliation index generalizes the Le Roux index for Brouwer homeomorphisms within the context of foliated Brouwer theory developed in \cite{lec1} (see Section \ref{sec:foliated_brouwer}).

\begin{thmx}\label{thmx:III-A}
    Let $\F$ be a transverse foliation of a Brouwer homeomorphism $f$, and let $\Gamma_1$ and $\sspc\Gamma_2$\break be proper transverse trajectories of two orbits $\O_{\sspc 1}$ and $\O_{\sspc 2}$ of $f$, respectively. Then, it holds
    $$\textup{Ind}(\F,\Gamma_1,\Gamma_2) = \textup{Ind}(f,\O_{\sspc 1},\O_{\sspc 2}).$$
\end{thmx}

\section[Background and notations]{
\mycomment{-0.7cm}Background and notations}

\mycomment{-0.5cm}

\subsection{Whitney's parametrization of $C^{\sspc 0}\sspc$-foliations} 

Let $\{q_i\}_{i \in \mathbb N}$ be a countable and dense subset of $\R^2$. For each $i \in \mathbb N$, let $m_i: \R^2 \longrightarrow (0,1]$ be the function that maps $p \in \R^2$ to the value
\begin{equation*}
    m_i(p) = \frac{1}{1 + d(p, q_i)},
\end{equation*}
where $d$ is the Euclidean distance in $\R^2$. 
Let $\P^*(\R^2)$ be the set of non-empty subsets of $\R^2$, and define the function $\mu_i: \P^*(\R^2) \longrightarrow [0,1]$ given by the expression
\begin{equation*}
    \mu_i (A) = \sup_{x,y \in A}\left|m_i(x) - m_i(y)\right|.
\end{equation*}
Finally, we can define the function $\mu: \P^*(\R^2) \longrightarrow [0,1]$ as
\begin{equation*}
    \mu(A) = \sum_{i \in \mathbb N} \frac{\mu_i(A)}{2^i}.
\end{equation*}
This function $\mu$ is called the \textit{Whitney function}, and it can be used to measure the ``size'' of subsets of $\R^2$. Its main properties are summarized in the following theorem.

\begin{theorem*}[Whitney, \cite{Whitney}]\label{thm:whitney}
    The Whitney function $\mu$ satisfies the following properties:
    \begin{itemize}[leftmargin=1.5cm]
        \item[\textbf{\textup{(i)}}] $\mu(A) = 0$ if, and only if, $A$ has only one element.
        \item[\textbf{\textup{(ii)}}] $\mu(\overline{\rule{0cm}{0.31cm}A\sspc}) = \mu(A)$.
        \item[\textbf{\textup{(iii)}}] If $A\subset A'$, then $\mu(A) \leq \mu(A')$.
        \item[\textbf{\textup{(iv)}}] If $\overline{\rule{0cm}{0.31cm}A\sspc} \subset A'$, then $\mu(A) < \mu(A')$.
        \item[\textbf{\textup{(v)}}] Let $(A_n)_{n \geq 0}$ be a sequence of compact subsets of $\R^2$ that converges in the Hausdorff metric to a compact subset $A\subset \R^2$. Then, we have the convergence $\mu(A_n) \longrightarrow \mu(A)$.
    \end{itemize}
\end{theorem*}

We refer to the original paper \cite{Whitney} for a proof of this theorem. In that paper, Whitney uses the function $\mu$ to parametrize topological (also called $C^{\sspc 0}\sspc$) foliations of $\R^2$. In this section, we use the same function to parametrize these foliations of $\R^2$, but with a different approach.

\mycomment{-0.9cm}
Let $\F$ be an oriented topological foliation of $\R^2$. Recall that, for any $z \in \R^2$, we denote:
\begin{itemize}[leftmargin=1.5cm]
    \item $\phi_z$ the leaf of $\F$ passing through $z$,
    \item $\phi_z^+$ the sub-halfline of $\phi_z$ that starts at $z$, and includes $z$,
    \item $\phi_z^-$ the sub-halfline of $\phi_z$ that ends at $z$, and includes $z$,
    \item $\phi_{z,z'}$ the subarc of $\phi_z$ that joins $z$ to a point $z' \in \phi_z^+$.
\end{itemize}
\mycomment{0.3cm}

\noindent Let $\tau : \mathbb \R^2 \longrightarrow (0,1]$ be the function defined by the expression
\mycomment{-0.1cm}
    \begin{equation*}
        \tau(z) = \sup_{z' \in \phi_z^+} \mu \left( \phi_{z,z'} \right) = \mu\bigl(\phi_z^+\bigr).
    \end{equation*}

\begin{lemma}\label{tau}
    The function $\tau$ is lower semi-continuous.
\end{lemma}

\mycomment{-0.2cm}
\begin{proof}
    Let $(z_n)_{n \geq 0}$ and $(z_n')_{n \geq 0}$ be sequences in $\R^2$ satisfying $z_n' \in \phi_{z_n}^+$ and converging to points $z$ and $z'$ in $\R^2$, respectively, which we assume to also satisfy the condition $z' \in \phi_z^+$. 
    Consider a neighborhood $V$ of $\phi_{z,z'}$ that trivializes the foliation $\F$ around the segment $\phi_{z,z'}$. Let $h: V \longrightarrow [0,1]^2$ be the corresponding trivializing map, mapping $\F\vert_V$ to the horizontal foliation of $[0,1]^2$. Up to taking a subsequence, we can assume that $z_n, z_n' \in V$ for all $n \geq 0$. Within the trivially foliated rectangle $[0,1]^2$, we have that $h(\phi_{z_n,z_n'})$ converges to $h(\phi_{z,z'})$ in the Hausdorff metric. Since $h^{-1}$ is uniformly continuous on the square $[0,1]^2$, this implies that $\phi_{z_n,z_n'}$ converges to $\phi_{z,z'}$ in the Hausdorff metric as well. By item (v) of Whitney's theorem, we have that 
    $\mu(\phi_{z_n,z_n'}) \longrightarrow \mu(\phi_{z,z'})$. This, together with the definition of $\tau$, yields
    $$\liminf_n \tau(z_n) \geq \liminf_n \mu(\phi_{z_n,z_n'}) = \mu(\phi_{z,z'}).$$

    Next, we observe that for any point $z''\in \phi_z^+$, there exists a sequence $(z_n'')_{n \geq 0}$ in $\R^2$ that satisfies 
    $z_n'' \in \phi_{z_n}^+$ and $z_n'' \longrightarrow z''$. This means that we can apply the inequality above to every point in $\phi_z^+$. In other words,
    $\liminf_n \tau(z_n) \geq \mu(\phi_{z,z''})$ for every $z'' \in \phi_z^+.$
    This implies that 
    $$\tau(z) = \sup_{z'' \in \phi_z^+} \mu(\phi_{z,z''}) \leq \liminf_n \tau(z_n).$$
    This proves that $\tau$ is lower semi-continuous.
\end{proof}

Let us define the set 
\mycomment{-0.2cm}
$$ W = \{(x,t) \in \R^2\times [0,1) \mid 0\leq t <\tau(x)\},$$
and consider the projection $\text{pr}_{\R^2}:\R^2\times[0,\infty)\longrightarrow \R^2$ given by $\text{pr}_{\R^2}(x,t) = x$.

\begin{lemma}
    There exists a homeomorphism $$g:W\longrightarrow \R^2\times [0,\infty)$$ that fixes the base space $\R^2$, meaning, it satisfies $\textup{pr}_{\R^2}\circ g = \textup{Id}_{\R^2}$.
\end{lemma}

\begin{proof}
    For any $k \in \mathbb N$, consider the convolution function
    $$\tau_k(x):=(1-1/k)\inf_{y \in \R^2} \left\{ \tau(y) + k \cdot d(x,y)\right\}.$$
    Observe that, each function $\tau_k:\R^2 \longrightarrow (0,1]$ satisfies the following properties:
    \begin{itemize}
        \item $\tau_k$ is $k$-Lipshitz continuous.
        \item $\tau_k(x) < \tau_{k+1}(x) < \tau(x)$, for every $x \in \R^2$.
        \item $\lim_{k\to\infty} \tau_k(x) = \tau(x)$, for every $x \in \R^2$.
    \end{itemize}

    \mycomment{0.2cm}
    \noindent Now, for each $k \in \mathbb N$, we define the set
    $$W_k = \{(x,t) \in \R^2\times [0,1] \mid 0\leq t \leq\tau_k(x)\}.$$
    Since $\tau_k$ is continuous, there exists a homeomorphism
    $$g_k:W_k\longrightarrow \R^2\times [0,k]$$ that maps $\text{graph}_{\spc \tau_k} = \{(x,\tau_k(x)) \in \R^2\times [0,k] \mid x \in \R^2\}$ to $\R^2\times \{k\}$ while fixing the base space $\R^2$, meaning that, $\text{pr}_{\R^2}\circ g_k = \text{Id}_{\R^2}$. Additionally, we can construct such maps $g_k$ in such a way that they are compatible with each other, meaning that 
    $$ g_{k+1}\vert_{W_k} = g_k\sspc, \quad \forall k \in \mathbb N.$$
    Since $W = \bigcup_{k \in \mathbb N} W_k$, we can consider the map $g:W\longrightarrow \R^2\times [0,\infty)$ defined by the property
    $$g\vert_{W_k} = g_k\sspc, \quad \forall k \in \mathbb N.$$
    Finally, to conclude the proof, we need to show that $g$ is a homeomorphism. Indeed, note that any point $(x,t) \in W$ admits a neighborhood $U$ such that $U\subset W_k$ for some $k \in \mathbb N$. Since $g_k$ is a homeomorphism, it follows that $g$ is a local homeomorphism. Since $g$ is also a bijection, it is a homeomorphism. Moreover, since $g_k$ fixes the base space $\R^2$, we have that $g$ also fixes the base space $\R^2$. This concludes the proof.
\end{proof}

Finally, we define the set
$$\overline{\rule{0cm}{0.31cm}W_\F\sspc} := \{(z,z') \in \R^2\times \R^2 \mid z' \in \phi_z^+\}.$$
Observe that, by defining the diagonal $\Delta = \{(z,z) \in \R^2\times \R^2\}$, the set $W_\F:=\overline{\rule{0cm}{0.31cm}W_\F\sspc} \setminus \Delta$ corresponds to the interior of the set $\overline{\rule{0cm}{0.31cm}W_\F\sspc}$. These sets are described in the following lemma.

\begin{lemma}\label{lem:parametrization}
    There exists a homeomorphism
    $$ \varphi:\R^2\times [0,\infty) \longrightarrow \overline{\rule{0cm}{0.31cm}W_\F\sspc}$$
    that satisfies $\varphi(z,0) = (z,z)$ and $\varphi(z,t) \in \phi_z^+$, for every $z \in \R^2$ and $t \in [0,\infty)$.
\end{lemma}

\begin{proof}
    We define $\varphi$ as the map satisfying
    $$\varphi(z,t) = (z, z'_t),$$
    where $z'_t$ is the point in $\phi_z^+$ determined by the equation $g(z, \mu(\phi_{z,z'_t})) = (z, t)$. From the properties of the function $\mu$ described in Whitney's theorem, we know that $\varphi$ is a well-defined bijection between $\R^2\times [0,\infty)$ and $\overline{\rule{0cm}{0.31cm}W_\F\sspc}$.
    To show that $\varphi$ is a homeomorphism, we need to prove that it is a local homeomorphism.

    Let $(z,z') \in \overline{\rule{0cm}{0.31cm}W_\F\sspc}$, and consider a neighborhood $V\subset \R^2$ of the subarc $\phi_{z,z'}$ that trivializes the foliation $\F$ around this subarc. This allows us to define an open neighborhood of $(z,z')$ within $\overline{\rule{0cm}{0.31cm}W_\F\sspc}$, given by
    $$V_\F = \{(x,x') \in \overline{\rule{0cm}{0.31cm}W_\F\sspc} \mid x, x' \in V\}.$$
    Using an argument similar to the one used in the proof of Lemma \ref{tau}, we can use the properties of the function $\mu$ to show that $\varphi^{-1}(V_\F)$ is an open subset of $\R^2\times [0,\infty)$, and that 
    $$\varphi^{-1}\vert_{V_\F}: V_\F \longrightarrow \varphi^{-1}(V_\F)$$
    is a homeomorphism. This proves that $\varphi$ is a local homeomorphism, and since it is also a bijection, it is a homeomorphism. Finally, we observe that $\varphi(z,0) = (z,z)$ and $\varphi(z,t) \in \phi_z^+$ for every $z \in \R^2$ and $t \in [0,\infty)$, which concludes the proof.
\end{proof}

 The map $\varphi$ in Lemma \ref{lem:parametrization} is of called the \textit{Whitney parameterization} of the foliation $\F$.

\mycomment{0.3cm}

\subsection{Serre fibrations}

The definitions and results in this section follow those in Tammo Tom Dieck's book on Algebraic Topology \cite{tomDieck2008algebraicTopology}.

Let $E$ and $B$ be topological spaces. A (continuous) map $p:E\longrightarrow B$ is said to satisfy the \textit{homotopy lifting property} (HLP) for a topological space $X$ if
\begin{itemize}
    \item for any homotopy $H:X\times [0,1]\longrightarrow B$,
    \item and for any continuous map $\widetilde{H}_0:X\longrightarrow E$ satisfying $H_0 := H\vert_{X\times \{0\}} = p\circ \widetilde{H}_0$,
\end{itemize}
\mycomment{0.2cm}
there exists a homotopy $\widetilde{H}:X\times [0,1]\longrightarrow E$ that satisfies $H = p\circ \widetilde{H}$.

\begin{definition}[Serre fibrations] A map $p:E\longrightarrow B$ is said to be a \textit{Serre fibration} if it satisfies the homotopy lifting property (HLP) for every CW-complex $X$.
\end{definition}

\begin{remark}
    The definition of Serre fibration is equivalent to the one given in \cite{tomDieck2008algebraicTopology}, which demands Serre fibrations to satisfy HLP for every cube $I^n = [0,1]^n$, for $n\geq 0$.
\end{remark}

The notion of Serre fibration is a generalization of Fiber Bundle. In fact, every fiber bundle is a Serre fibration, but not every Serre fibration is a fiber bundle. The difference is that the fibers of a fiber bundle are given by a fixed topological space, while the fibers of a Serre fibration can vary from point to point in the base space, as long as they are homotopy equivalent to each other.

In the math literature, the word fibration is often used to refer to Hurewicz fibrations, which are maps $p:E\longrightarrow B$ that satisfy the homotopy lifting property (HLP) for every topological space $X$. However, in this work, we opt to work with Serre fibrations, as they are easier to manipulate and suitable for the investigation of homotopy groups.

The following theorem states that being a Serre fibration is a local property.

\begin{theorem}\textup{\cite[Theorem 6.3.3]{tomDieck2008algebraicTopology}}\label{thm:local_serre_fibration}
    Let $p:E\longrightarrow B$ be a continuous map and let $\mathcal V$ be an open cover of $B$. If the map $p\vert_{p^{-1}(V)}:p^{-1}(V)\longrightarrow V$ is a Serre fibration for every $V\in \mathcal V$, then $p$ is itself a Serre fibration.
\end{theorem}

The main property of Serre fibrations that we will use in this work is a consequence of the long exact sequence of homotopy groups associated to a Serre fibration, stated below.

\begin{theorem}\textup{\cite[Theorem 6.3.2]{tomDieck2008algebraicTopology}}\label{thm:long_exact_sequence}
    Let $e_0 \in E$ and $b_0 = p(e_0) \in B$ be points in the total space and the base space of a Serre fibration $p:E\longrightarrow B$. Denote by $F = p^{-1}(b_0)$ the fiber of $p$ over the point $b_0$, and by $i:F\longrightarrow E$ the inclusion map. Then, we have a long exact sequence of homotopy groups of the form
    $$\cdots \longrightarrow \pi_n(F) \xrightarrow{i_*} \pi_n(E) \xrightarrow{p_*} \pi_n(B) \xrightarrow{\partial} \pi_{n-1}(F) \longrightarrow \cdots.$$
\end{theorem}

In the theorem above, the map $p_*$ is the homomorphism induced by the map $p:E\longrightarrow B$, and $\partial$ is the connecting homomorphism (see Section 6.3 in \cite{tomDieck2008algebraicTopology} for more details). 

Observe that, in the case $\pi_1(B,b_0) = \pi_1(F,e_0) = 0$, we obtain the trivial exact sequence
$$0  \longrightarrow \pi_1(E,e_0) \longrightarrow 0,$$
which implies that $\pi_1(E,e_0) = 0$. This yields the following corollary.

\begin{corollary}\label{cor:simpli-connected-fibrations}
    Let $p:E\longrightarrow B$ be a Serre fibration, and assume that $B$ and $F= p^{-1}(b_0)$ are simply-connected spaces. Then, the total space $E$ is also simply-connected.
\end{corollary}

\subsection{Foliated Brouwer Theory}\label{sec:foliated_brouwer}
Let \( f:\mathbb{R}^2 \longrightarrow \mathbb{R}^2 \) be a \textit{Brouwer homeomorphism}, that is, an orientation-preserving and fixed-point free homeomorphism of the plane.

The orbit of a point \( x \in \mathbb{R}^2 \) under \( f \) is defined as the set
$$ \mathcal{O}(f, x) = \{f^n(x) \mid n \in \mathbb{Z}\} \subset \mathbb{R}^2.$$
Sometimes, it may be useful to denote the set of orbits of \( f \) by $\orb$. As proved in \cite{Brouwer}, every orbit of $f$ is wandering and, thus, it has an empty limit set. This means that every orbit of $f$ is the image of a proper embedding of $\mathbb{Z}$ into the plane.

A \textit{Brouwer line} for \( f \) is an oriented topological line \( \lambda\) on \( \mathbb{R}^2 \) that satisfies
$$ f(L(\lambda)) \subset L(\lambda) \quad \text{ and } \quad f^{-1}(R(\lambda)) \subset R(\lambda),$$
where \( L(\lambda) \) and \( R(\lambda) \) denote the left and right sides of \( \lambda \). The classical \textit{Brouwer translation theorem} \cite{Brouwer} asserts that through every point of the plane passes a Brouwer line for \( f \).

\begin{theorem*}[Le Calvez, \cite{lec1}] There exists an oriented topological foliation \(\F\) of the plane such that every leaf $\phi \in \F$ is a Brouwer line of $f$.
\end{theorem*}

Any such foliation of the plane by Brouwer lines of \( f \) is called a \textit{transverse foliation} of \( f \). It is worth noting that a Brouwer homeomorphism admits several transverse foliations, which  may exhibit significantly different structures and non-homeomorphic leaf spaces, for instance.

Let us fix a transverse foliation $\F$ of a Brouwer homeomorphism \( f \). In this contex, one can show that every point \( x \in \mathbb{R}^2 \) may be connected to its image \( f(x) \) by a path $\gamma_x:[0,1]\longrightarrow \R^2$ that is positively transverse to the foliation $\F$, in the sense that $\gamma_x$ intersects each leaf of $\phi \in \F$ at most once, always from right $R(\phi)$ to left $L(\phi)$. By concatenating the paths $\gamma_{f^n(x)}$ for all integers $n\in \Z$, we obtain a positively transverse path 
$$ \Gamma := \prod_{n\in \Z} \gamma_{f^n(x)},$$
that connects the entire orbit $\O = \O(f,x)$, called a \textit{transverse trajectory} of $\O$ (relative to $\F$). 
It is worth noting that any orbit $\O$ admits uncountably many transverse trajectories, which depend on the choice of paths $\{\gamma_{f^n(x)}\}_{n \in \Z}$. However, all these transverse trajectories cross the same leaves of $\F$. We say an orbit of $f$ \textit{crosses} a leaf of $\F$ if its transverse trajectories do so.

We note that, if the choice of paths $\{\gamma_{f^n(x)}\}_{n \in \Z}$ is locally-finite on the plane, then the resulting transverse trajectory $\Gamma$ is a topological line, and it is called a \textit{proper transverse trajectory} of  $\O$. In a previous work \cite{schuback1}, we proved the following result.

\begin{theorem}\textup{\cite{schuback2}}\label{thm:proper} Every orbit of $f$ admits a proper transverse trajectory relative to $\F$.
\end{theorem}

\subsection{Homotopy Brouwer Theory}\label{sec:homotopy_brouwer_theory}
Let \( f:\mathbb{R}^2 \longrightarrow \mathbb{R}^2 \) be a \textit{Brouwer homeomorphism}, and let $\OO = \{\O_{\sspc 1},\ldots,\O_{\sspc r}\}$ be a finite collection of distinguished orbits of $f$.  By abuse of notation, we may denote $\OO := \bigcup_{i=1}^r\O_{\sspc i}\subset \R^2$.
The isotopy class of $f$ relative to the set $\OO$, denoted by $\class{f,\sspc \OO}$,\break is referred to as a \emph{Brouwer mapping class relative to $r$-orbits}. Alternatively, $\class{f,\sspc \OO}$ can be defined as the isotopy class of the map $f\vert_{\R^2\setminus \OO}$ on the surface $\R^2\setminus \OO$, which is homeomorphic to $\R^2\setminus \Z$, because $\OO$ is an infinite closed discrete subset of $\R^2$ (see \cite[Proposition 1.5]{BeguinCrovisierLeRoux2020}).

Two Brouwer mapping classes \(\class{f,\OO}\) and \(\class{f',\OO'}\) are conjugate, denoted \(\class{f,\OO} \simeq \class{f',\OO'}\), if there exists an orientation-preserving homeomorphism $h:\R^2\longrightarrow \R^2$ that satisfies
$$ h(\OO^\pp)=\OO \quad \text{ and } \quad \bigl[\spc f,\sspc \OO\spc\bigr] = \bigl[\spc h\circ f' \circ h^{-1},\sspc \OO\spc\bigr].$$
Our notion of conjugacy follows the one in \cite{BAVARD_2017}, which is the most natural one when seeing Brouwer mapping classes as elements of $\mcg(\R^2\setminus \Z)$. In \cite{handel99}, Handel includes conjugacies by orientation-reversing homeomorphisms, and in \cite{LEROUX_2017}, Le Roux only considers orientation-preserving conjugacies that preserve the indexation of the orbits. These differences are not a problem, as the statements of the result can be adapted depending on the context.

For each $i,j \in \Z$, we denote $\sspc \Z_{\sspc i} := \Z \times \{i\}$ and $\Z_{\sspc i,j} := \Z \times \{i,j\}$. Let $T$ be the horizontal translation, that moves each point in the plane one unit to the right, and the map $R$ is the time-one map of the standard Reeb flow, that moves $\Z_1$ the right and $\Z_2$ one to the left.

\mycomment{0.3cm}
\begin{figure}[h!]
    \center 
    \hspace*{-0.4cm}\begin{overpic}[width=4cm, height=2.5cm, tics=10]{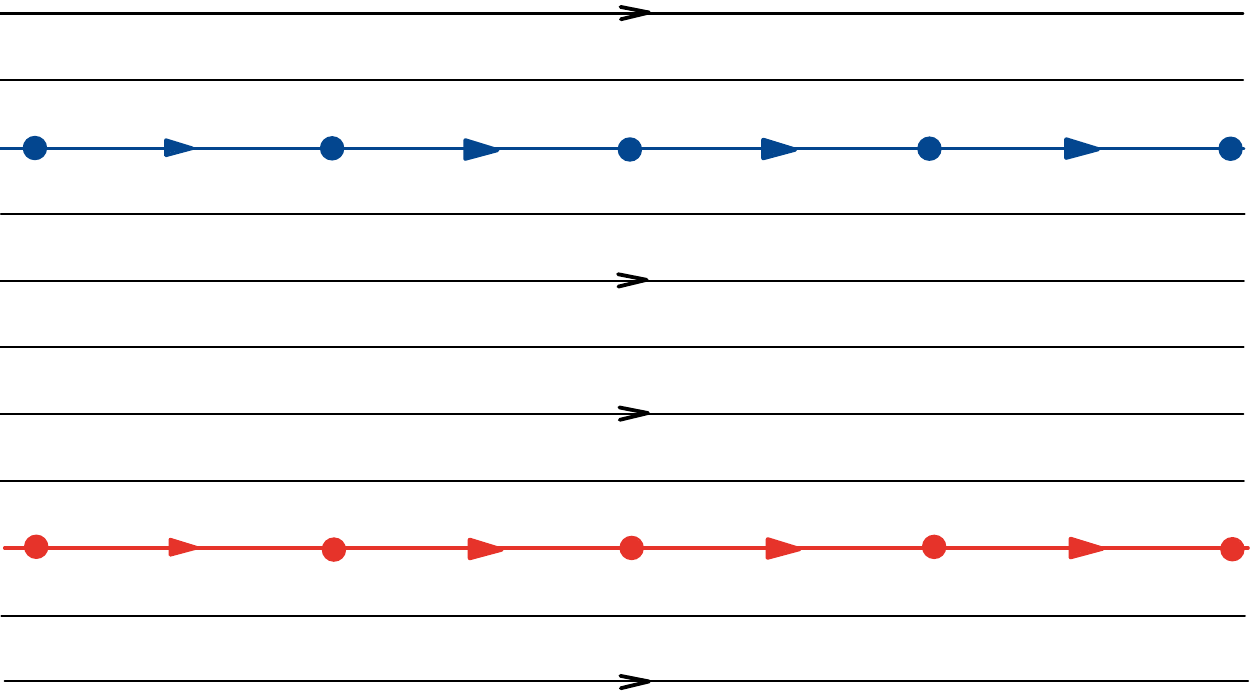}
        \put (-15,56) {{\color{black}\Large$\displaystyle T$}}
        \put (105,45) {{\color{myBLUE}\large$\displaystyle \Z_2$}}
        \put (105,9.5) {{\color{myRED}\large$\displaystyle \Z_1$}}
\end{overpic}
\hspace*{2.9cm}\begin{overpic}[width=4cm, height=2.5cm, tics=10]{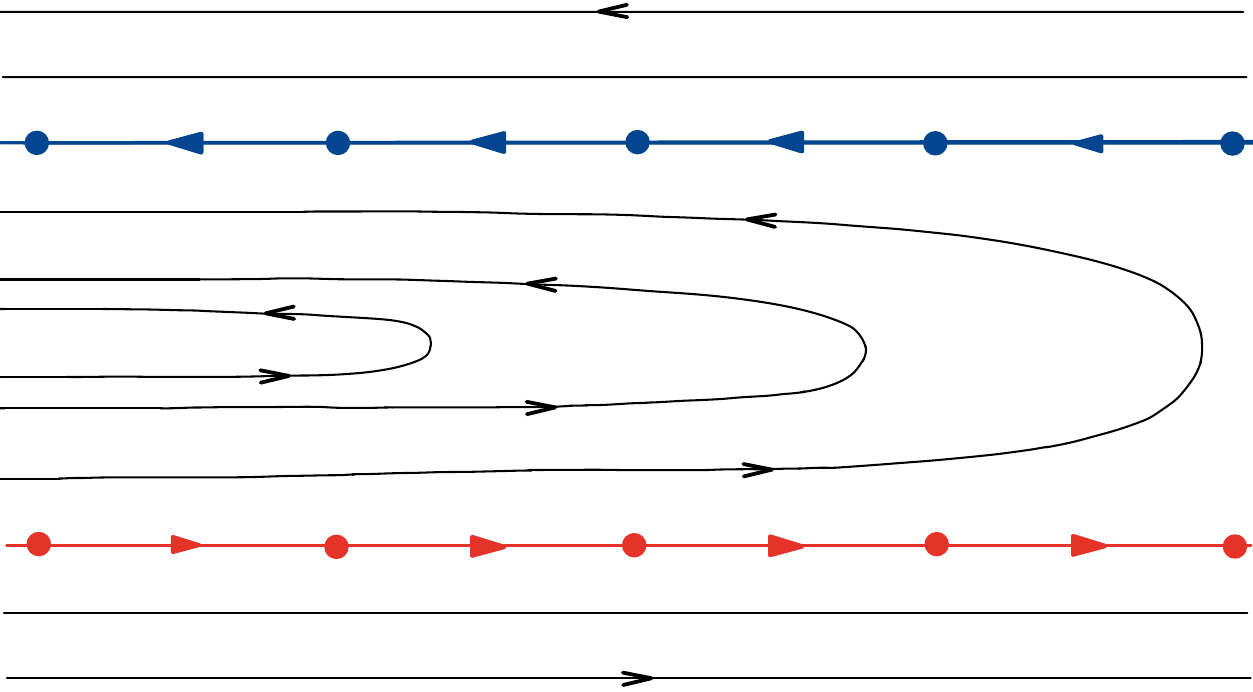}
    \put (-15,56) {{\color{black}\Large$\displaystyle R$}}
    \put (105,45) {{\color{myBLUE}\large$\displaystyle \Z_2$}}
        \put (105,9.5) {{\color{myRED}\large$\displaystyle \Z_1$}}
\end{overpic}
\end{figure}

Handel's classification theorem for Brouwer mapping classes relative to two orbits \cite{handel99} proves that there exists an orientation-preserving homeomorphism $h:\R^2\longrightarrow \R^2$ that maps $h(\O_{\sspc i}) = \Z_i$, for each $i\in\{1,2\}$, and conjugates $\big[\spc f,\sspc\{\O_{\sspc 1},\O_{\sspc 2}\}\spc\big]$ to one of the following classes:
\mycomment{-0.0cm}
$$ \bigl[\spc T\sspc,\spc\Z_{1,2}\sspc\bigr]\sspc, \quad \class{T^{-1},\spc\Z_{1,2}}\sspc, \quad \bigl[\spc R\sspc,\spc\Z_{1,2}\sspc\bigr]\sspc, \quad \class{R^{-1},\spc\Z_{1,2}}\sspc.$$
The set of all such homeomorphisms $h$ for $f$, $\O_1$ and $\O_2$ is denoted by $\textup{Handel}(\sspc f,\O_{\sspc 1},\O_{\sspc 2} )$.

In a previous work \cite{schuback2}, we have combined the Foliated and Homotopy frameworks in Brouwer Theory. This approach allows us to study Brouwer mapping classes through transverse foliations, leading to new insights and results. 

\begin{theorem}\textup{\cite{schuback2}}\label{thm:free} Let $\F$ be a transverse foliation of $f$. If there exists a leaf $\phi \in \F$ that is crossed by both orbits $\O_{\sspc 1}$ and $\O_{\sspc 2}$, then there exists an isotopy $(f_t)_{t\in [0,1]}$ relative to $\O_{\sspc 1} \cup \O_{\sspc 2}$,\break that satisfies $\textup{Fix}(f_t) = \varnothing$ for every $t \in [0,1]$, and joins $f_0 = f$ to a map $f_1$ conjugate to $T$.
\end{theorem}


\mycomment{-0.9cm}
\subsection{The Le Roux index for Brouwer homeomorphisms}

For the reader’s convenience, we repeat here the definition of the Le Roux index, previously given in the introduction.

\mycomment{0.2cm}

Let $f:\R^2\longrightarrow \R^2$ be a Brouwer homeomorphism, and let $\O_{\sspc 1}$ and $\O_{\sspc 2}$ be two orbits of $f$.\break Let $X_f$ be the $C^0$-vector field on the plane $X_f(x) := f(x)-x$, and let $h\in\text{Handel}(f,\O_{\sspc 1},\O_{\sspc 2})$. Observe that the vector field $X_{h\sspc\circ\sspc f \sspc\circ\sspc h^{-1}}$
is horizontal on $\Z_1$ and $\Z_2$. Thus, the winding number\break $\text{Wnd}(X_{h\sspc\circ\sspc f \sspc\circ\sspc h^{-1}},\alpha)$ along any path $\alpha$ joining $\Z_1$ to $\Z_2$ must be an integer multiple of $1/2$.\break In \cite{LEROUX_2017}, Le Roux shows that this value is independent of the choice of $h\in \text{Handel}(f,\O_{\sspc 1},\O_{\sspc 2})$ and $\alpha$, and therefore defines an index $\ind(\sspc f,\O_{\sspc 1},\O_{\sspc 2} )$, called the \textit{Le Roux index}.
\vspace*{0.2cm}

\begin{figure}[h!]
    \hspace*{-1.2cm}\begin{overpic}[width=8.5cm, height=3.6cm, tics=10]{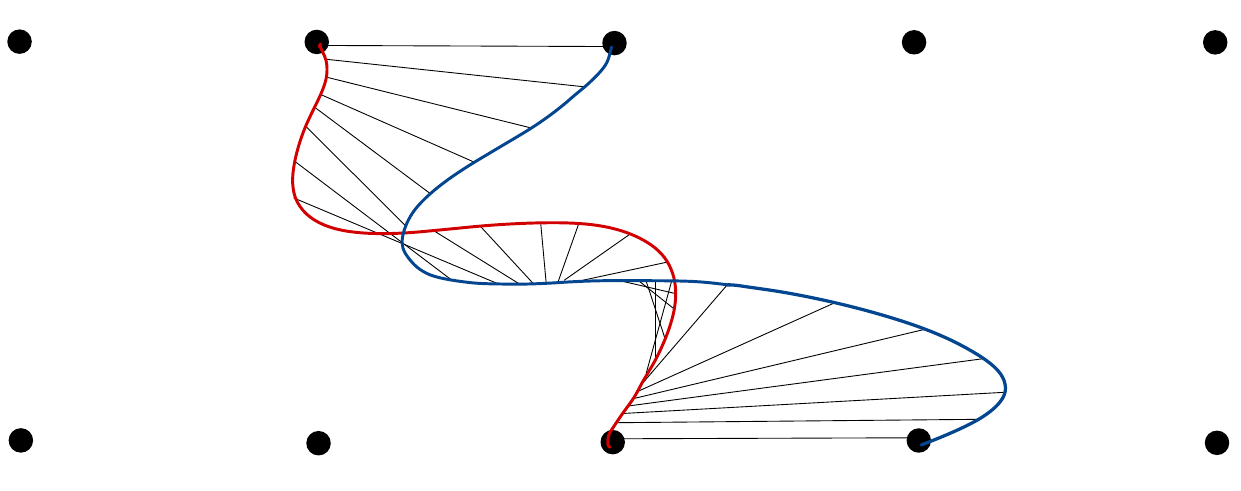}
        \put (28,43.8) {{\color{black}\large$\displaystyle X_{h\sspc\circ\sspc f \sspc\circ\sspc h^{-1}}$}}
        \put (105,37.2) {{\color{black}\large$\displaystyle \Z_2 = h(\O_2)$}}
        \put (105,1.8) {{\color{black}\large$\displaystyle \Z_1=h(\O_1)$}}
        \put (60,20) {{\color{myBLUE}\large$\displaystyle h\circ f \circ h^{-1} \circ \alpha$}}
        \put (18,26) {{\color{myRED}\large$\displaystyle \alpha$}}
\end{overpic}
\end{figure}

\vspace*{-0.2cm}
\noindent This index extends the classical Poincaré-Hopf index to the context of Brouwer homeomorphisms. 

It is worth mentioning that this index is not a conjugacy invariant for Brouwer mapping classes relative to two orbits. However, from its definition, we note that it is invariant under fixed point-free isotopies relative to the two orbits.

\section{Topological angle function}


\mycomment{0.5cm}

The Khalimsky topology on $\Z$ is the coarsest topology on $\Z$ making all singletons of the form $\{2k+1\}$ and all triplets of the form $\{2k-1, 2k, 2k+1\}$ open sets, where $k \in \Z$. This topology is also known as the \textit{digital topology}. The main feature of this topology is that odd integers admit an isolating neighborhood, while even integers do not. In \cite{CalvezTwist}, Le Calvez uses this topological feature to define a notion of topological angle function for radial foliations of the annulus, with the goal of providing an abstract point of view on twist maps of the annulus. In this chapter, we will use this same topological feature to describe the relative position of points with respect to a non-singular planar foliation.

Let $\ZZ=\{\sspc-\dot{1}\sspc, \spc \dot 0\sspc, \spc \dot 1\sspc, \spc \dot 2\sspc\}$ be the set of integers $\text{mod } 4$, endowed with the topology $\mathcal K$ induced by the quotient $\sspc\text{mod}_4: \Z \longrightarrow \ZZ$. More precisely, this topology is given by
$$ \mathcal K :=\left\{ \spc \varnothing\sspc, \spc \{\sspc-\dot{1}\sspc\}\sspc, \spc \{\sspc\dot 1\sspc\}\sspc, \spc \{\sspc-\dot{1}\sspc, \spc \dot 1\sspc\}\sspc, \spc \{\sspc-\dot{1}\sspc, \spc \dot 0\sspc, \spc \dot 1\sspc\}\sspc, \spc \{\sspc\dot{1}\sspc, \spc \dot 2\sspc, \spc -\dot {1}\sspc\}\sspc, \spc \{\sspc-\dot{1}\sspc, \spc \dot 0\sspc, \spc \dot 1\sspc, \spc \dot 2\sspc\}  \spc \right\}.$$
Let $\mathscr{F}$ be the space of all oriented topological (non-singular) foliations of the plane. Moreover, consider the diagonal $\Delta = \{(z,z) \in \R^2\times \R^2\}$ and the configuration space
$$ \R^2\times \R^2\setminus \Delta = \{(z,z') \in \R^2\times \R^2 \mid z\neq z'\}.$$

We introduce the topological angle function
$\sspc\dot\theta:(\R^2\times \R^2\setminus \Delta) \times \mathscr{F} \longrightarrow \ZZ\spc$, defined as
\mycomment{0.2cm}
\begin{equation*}
    \Tp(\sspc z\sspc , \sspc z'\sspc , \sspc \F\sspc ) = \begin{cases}
        \ \ \dot 0 \ &   \textup{if } z' \in \phi^+_z, \\
        \ \ \dot{1}\ &  \textup{if } z' \in L(\phi_z), \\
        \ \ \dot{2}\ &   \textup{if } z' \in \phi^-_z,\\
        \ -\dot{1}\ &   \textup{if } z' \in R(\phi_z).
    \end{cases}
\end{equation*}

\mycomment{-0.7cm}
\noindent Roughly speaking, the value $\Tp(z,z',\F)$ computes the relative position of a point $z'$ with respect to a basepoint $z$ and the foliation $\F$. As we will see throughout this work, this function is very useful for extracting topological information about the foliation $\F$.

\begin{lemma}\label{lemma:continuous}
    The following properties hold:
    \begin{itemize}
        \item[\textup{\textbf{(i)}}] For any orientation-preserving homeomorphism $h:\R^2\longrightarrow \R^2$, it holds that\
        \mycomment{-0.12cm}
        $$\Tp(\sspc z\sspc , \sspc z'\sspc , \sspc h\F\sspc ) = \Tp(\sspc h^{-1}(z)\sspc , \sspc h^{-1}(z')\sspc , \sspc\F\sspc ),$$

        \mycomment{-0.24cm}
        \noindent
        where $h\F$ is the foliation obtained by applying $h$ to each leaf of $\F$.

        \item[\textup{\textbf{(ii)}}] The function $\spc\dot\theta(\hspace*{0.08cm} \cdot \hspace*{0.08cm}, \hspace*{0.08cm}\cdot \hspace*{0.08cm}, \F )$ is continuous on the entire configuration space $\R^2\times \R^2\setminus \Delta$ if, and only if, the foliation $\F$ is trivial.
    \end{itemize}
\end{lemma}

\mycomment{-0.22cm}
\begin{proof}
    The first item follows from the definition of $\Tp$. To prove one part of item (ii), we can use item (i) to reduce the problem to the case of a horizontal foliation of $\R^2$.

    Let $\H$ be the horizontal foliation of $\R^2$, with each leaf oriented positively in the $x$-direction. Consider the function $\sspc\sigma:\R/2\pi\Z \longrightarrow \ZZ\spc$ defined by
    \mycomment{0.1cm}
    \begin{equation*}
    \sigma(t)= \begin{cases}
        \ \ \dot 0 \ &   \textup{if } t = 0, \\
        \ \ \dot{1}\ &  \textup{if } t \in (0, \pi), \\
        \ \ \dot{2}\ &   \textup{if } t =\pi,\\
        \ -\dot{1}\ &   \textup{if } t \in (\pi, 2\pi). \\
    \end{cases}
\end{equation*}

\noindent Observe that $\sigma$ is continuous with respect to the topology $\mathcal K$ on $\ZZ$.

To prove the ``if'' implication of item (ii), it suffices to note that
\mycomment{-0.12cm}
$$ \Tp(\sspc z\sspc , \sspc z'\sspc , \H ) = \sigma(\spc\textup{Angle}\sspc(z'-z)\spc),$$

\mycomment{-0.23cm}
\noindent
where $\textup{Angle}\sspc(z'-z)$ is the angle between the vector $z'-z$ and the positive $x$-axis, measured in the counterclockwise direction. Since $\textup{Angle}\sspc(z'-z)$ is continuous in both variables $z$ and $z'$, the map $\Tp(\sspc z\sspc , \sspc z'\sspc , \H )$ is continuous on $\R^2\times \R^2\setminus \Delta$ with respect to the topology $\mathcal K$ on $\ZZ$.

To prove the ``only if'' implication, one should observe that, if the foliation $\F$ is not trivial, then it must have a Reeb component. This allows us to consider two sequences of points $(z_n)_{n \geq 0}$ and $(z_n')_{n \geq 0}$ in $\R^2$ satisfying $z_n' \in \phi^+_{z_n}$ and such that
\mycomment{-0.1cm}
$$ z_n \longrightarrow z \in \R^2 \quad \text{ and } \quad z_n' \longrightarrow z' \in \R^2,$$

\mycomment{-0.23cm}
\noindent
where $z$ and $z'$ lie on different leaves of $\F$. In other words, we have that $\Tp(\sspc z\sspc , \sspc z'\sspc , \F ) \in \{-\dot{1}, \dot{1}\}$, while $\Tp(\sspc z_n\sspc , \sspc z_n'\sspc , \F ) = \dot{0}$ for all $n\geq 0$. Since $\{-\dot{1}\}$ and $\{\sspc\dot{1}\sspc\}$ are open sets for $\mathcal K$, this implies that the function $\Tp(\sspc \cdot\sspc , \sspc \cdot\sspc , \F )$ is not continuous at the point $(z,z')$. This concludes the proof.
\end{proof}

\subsection{Rewriting the Le Roux index}

The topological angle function provides an alternative way to define the index of a Brouwer homeomorphism between two orbits, originally introduced by Le Roux in \cite{LEROUX_2017}. This alternative definition is particularly interesting, as it establishes a common language to describe the index of a Brouwer homeomorphism and the structure of its transverse foliations.

Let $f:\R^2\longrightarrow \R^2$ be a Brouwer homeomorphism, and define the set 
\begin{equation*}
    W_f = \{(z,f(z)) \in \R^2\times \R^2 \mid z \in \R^2\}.
\end{equation*}
Since $f$ has no fixed points, $W_f$ is a subset of $\spc \R^2\times \R^2\setminus \Delta\spc$. Moreover, note that the set $W_f$ is homeomorphic to $\R^2$ via the map $\text{graph}_f:\R^2\longrightarrow W_f$ given by $\text{graph}_f(z) = (z,f(z))$. In particular, this shows that $W_f$ is a simply-connected subset of $\R^2\times \R^2\setminus \Delta$.

Let $\H$ be the horizontal foliation of $\R^2$ with each leaf oriented positively in the $x$-direction.
Since $W_f$ is simply-connected and $\spc\dot\theta(\hspace*{0.08cm} \cdot \hspace*{0.08cm}, \hspace*{0.08cm}\cdot \hspace*{0.08cm}, \H )$ is continuous on $\R^2\times \R^2\setminus \Delta$ (Lemma \ref{lemma:continuous}), the lifting theorem (see Theorem 3.5.2 in \cite{tomDieck2008algebraicTopology}) tells us that the function $\spc\dot\theta(\hspace*{0.08cm} \cdot \hspace*{0.08cm}, \hspace*{0.08cm}\cdot \hspace*{0.08cm}, \H )\vert_{W_f}$ admits a lift $\Theta_f:\W_f\longrightarrow \Z$, which is unique up to the addition of an integer multiple of $4$. This setting is represented in the following diagram.

\begin{center}
\hspace*{0.6cm}\begin{tikzpicture}[scale=0.81, node distance=2.2cm]
    \node (Wf) at (0,0) {$W_f$};
    \node (Z4) at (4.3,0) {$\ZZ$};
    \node (Z) at (4.3,2.8) {$\mathbb{Z}$};
    
    \draw[->, thick] ([xshift=1pt]Wf.east) -- ([xshift=-1.5pt]Z4.west) node[midway, below=0.1cm] {$\spc\dot\theta(\hspace*{0.08cm} \cdot \hspace*{0.08cm}, \hspace*{0.08cm}\cdot \hspace*{0.08cm}, \H )$};
    \draw[->, thick] ([yshift=1pt]Wf.north east) -- ([yshift=-1pt]Z.south west) node[midway, above left=0.05cm] {$\Theta_f$};
    \draw[->, thick] ([yshift=-2.5pt]Z.south) -- ([yshift=2.5pt]Z4.north) node[midway, right=0.1cm] {$\text{mod}_4$};
    
    \node at (2.85,0.9) {\LARGE$\circlearrowright$};
\end{tikzpicture}
\end{center}

\mycomment{-0.4cm}
Recall that, as a consequence of Handel's classification of Brouwer mapping classes relative to two orbits (see \cite{handel99}), for any two orbits $\O_1$ and $\O_2$ of $f$, there exists an orientation-preserving homeomorphism $h:\R^2\longrightarrow \R^2$ 
such that $h(\O_{\sspc i}) = \Z_{\sspc i}\sspc$ for $i \in \{1,2\}$ and 
\mycomment{0.2cm}
$$ \class{h\circ f \circ h^{-1},\Z_{1,2}} \in \left\{\begin{array}{lr}
         \sspc\class{T\spc,\spc\Z_{1,2}}\sspc, 
        \quad\class{T^{-1},\sspc\Z_{1,2}}\sspc,\\
         \sspc\class{R\spc,\spc\Z_{1,2}}\sspc,
        \quad\class{R^{-1},\sspc\Z_{1,2}}\sspc.
        \end{array}\right\}
$$

\mycomment{0.1cm}
\noindent The set of all such homeomorphisms $h$ for $f$, $\O_1$ and $\O_2$ is denoted by $\textup{Handel}(f, \O_1,\O_2)$.

\begin{proposition}
    Let $f:\R^2\longrightarrow \R^2$ be a Brouwer homeomorphism, and let $\O_1$ and $\O_2$ be two orbits of $f$. Then, for any $h \in \textup{Handel}(f, \O_1,\O_2)$, $x_1 \in \Z_1$ and $x_2 \in \Z_2$, we have that
    $$\textup{Ind}(f,\O_1,\O_2) = \frac{1}{4}\Bigl(\sspc\Theta_F(x_2,F(x_2)) - \Theta_F(x_1,F(x_1))\sspc\Bigr),$$
    where we denote $F:=h\circ f \circ h^{-1}$.
\end{proposition}

\begin{proof}
    Recall that, as explained in the introduction, the index $\text{Ind}(f,\O_1,\O_2)$ is defined as the winding number of the vector field 
    $$X_{h \sspc \circ \sspc f \sspc \circ \sspc h^{-1}\sspc}(x) = h\circ f \circ h^{-1}(x) - x, \quad \forall x \in \R^2,$$
    over a path $\alpha:[0,1]\longrightarrow \R^2$ satisfying $\alpha(0) \in \Z_1$ and $\alpha(1) \in \Z_2$, for any $h\in \textup{Handel}(f, \O_1,\O_2)$.
    To be more precise, by considering the map $s:[0,1]\longrightarrow \R/2\pi\Z$ given by the expression
    $$ s(t) = \textup{Angle}\sspc\Bigl(\sspc h\circ f\circ h^{-1} (\alpha(t)) -\alpha(t)\sspc \Bigr),$$
    and fixing a lift $\widetilde{s}:[0,1]\longrightarrow \R$ of $s$ through the quotient map $\R \longrightarrow \R/2\pi\Z$, we have that
    $$\text{Ind}(f,\O_1,\O_2) = \frac{1}{2\pi}\bigl(\sspc\widetilde{s}\spc(1) - \widetilde{s}\spc(0)\sspc \bigr).$$
    This value is always an integer multiple of $1/2$, since $s(0)$ and $s(1)$ are always $0$ or $\pi$, and completely independent of the choice of $\alpha$ and $h$, as proven in \cite{LEROUX_2017}.

    Consider the path $A:[0,1]\longrightarrow \R^2\times \R^2\setminus \Delta$ given by
    $$ A(t) = \bigl(\sspc h\circ \alpha(t)\sspc, \sspc h\circ f\circ \alpha(t)\sspc\bigr).$$
    By denoting $F = h\circ f \circ h^{-1}$, we observe that the image of $A$ is contained in $W_F$, because
    $$ W_F = \{(h(z), h\circ f (z)) \mid z \in \R^2\}.$$
    By revisiting the relation between the map $\spc\dot\theta(\hspace*{0.08cm} \cdot \hspace*{0.08cm}, \hspace*{0.08cm}\cdot \hspace*{0.08cm}, \H )$ and the function $\sigma:\R/2\pi\Z \longrightarrow \ZZ$, defined in the proof of Lemma \ref{lemma:continuous}, we obtain the following commutative diagram.

\begin{center}
\hspace*{0.2cm}\begin{tikzpicture}[scale=0.8, node distance=2.2cm]
    \node (Wf) at (0,3) {$W_F$};
    \node (I) at (0,0) {$[0,1]$};
    \node (Z4) at (4.3,3) {$\ZZ$};
    \node (Z) at (4.3,0) {$\R/2\pi\Z$};
    
    \draw[->, thick] ([xshift=1pt]Wf.east) -- ([xshift=-1.5pt]Z4.west) node[midway, above=0.1cm] {$\spc\dot\theta(\hspace*{0.08cm} \cdot \hspace*{0.08cm}, \hspace*{0.08cm}\cdot \hspace*{0.08cm}, \H )$};
    \draw[->, thick] ([xshift=1pt]I.east) -- ([xshift=-1.5pt]Z.west) node[midway, below=0.1cm] {$s$};
    \draw[->, thick] ([yshift=1pt]Z.north) -- ([yshift=-1pt]Z4.south) node[midway, right=0.1cm] {$\sigma$};
    \draw[->, thick] ([yshift=1pt]I.north) -- ([yshift=-1pt]Wf.south) node[midway, left=0.1cm] {$A$};
    
    \node at (2.15,1.5) {\LARGE$\circlearrowright$};
\end{tikzpicture}
\end{center}

\mycomment{-0.3cm}
Consider the function $\widetilde{\sigma}:\R\longrightarrow \Z$ defined by the floor function $\widetilde{\sigma}(t) = \Big\lfloor\frac{2t}{\pi} \Big\rfloor$. Note that, the function $\widetilde{\sigma}$ satisfies $\widetilde{\sigma}(k\pi) = 2k$ for every $k\in \Z$ and, moreover, $\sspc\text{mod}_4\spc \circ \widetilde{\sigma} \spc=\spc \sigma \circ \sspc\text{mod}_{2\pi}\sspc$, as shown in the commutative diagram below.

\begin{center}
\hspace*{0.2cm}\begin{tikzpicture}[scale=0.8, node distance=2.2cm]
    \node (R) at (0,3) {$\Z$};
    \node (I) at (0,0) {$\R$};
    \node (Z4) at (4.3,3) {$\ZZ$};
    \node (Z) at (4.3,0) {$\R/2\pi\Z$};

    \draw[->, thick] ([xshift=1pt]R.east) -- ([xshift=-1.5pt]Z4.west) node[midway, above=0.1cm] {$\text{mod}_4$};
    \draw[->, thick] ([xshift=1pt]I.east) -- ([xshift=-1.5pt]Z.west) node[midway, below=0.1cm] {$\text{mod}_{2\pi}$};
    \draw[->, thick] ([yshift=1pt]Z.north) -- ([yshift=-1.5pt]Z4.south) node[midway, right=0.1cm] {$\sigma$};
    \draw[->, thick] ([yshift=1pt]I.north) -- ([yshift=-1.5pt]R.south) node[midway, left=0.1cm] {$\widetilde\sigma$};
    
    \node at (2.15,1.5) {\LARGE$\circlearrowright$};
\end{tikzpicture}
\end{center}

Combining both diagrams with those associated with the lifts $\Theta_F$ and $\widetilde{s}$, we can represent the setting through the following diagram.

\mycomment{-0cm}

\begin{center}
\hspace*{0.2cm}\begin{tikzpicture}[scale=0.84, node distance=2.2cm]
    \node (Wf) at (0,3.2) {$W_F$};
    \node (I) at (0,0) {$[0,1]$};
    \node (Z4) at (4.3,3.2) {$\ZZ$};
    \node (Z) at (4.3,0) {$\R/2\pi\Z$};
    \node (R) at (4.3,-3.2) {$\R$};
     \node (ZZZ) at (8,5.6) {$\Z$};
    
    \draw[->, thick]  ([xshift=-1.5pt]ZZZ.south west) --([xshift=1pt]Z4.north east) node[midway, above=0.15cm] {$\text{mod}_4 \quad \quad $};
    \draw[->, thick] ([xshift=1pt]Wf.east) -- ([xshift=-1.5pt]Z4.west) node[midway, above=0.1cm] {$\spc\dot\theta(\hspace*{0.08cm} \cdot \hspace*{0.08cm}, \hspace*{0.08cm}\cdot \hspace*{0.08cm}, \H )$};
    \draw[->, thick] ([xshift=1pt]I.east) -- ([xshift=-1.5pt]Z.west) node[midway, below=0.1cm] {$\ \ s$};
    \draw[->, thick] ([xshift=1pt]I.south east) -- ([xshift=-1.5pt]R.north west) node[midway, below left=0.05cm] {$\widetilde{s}$};
    \draw[->, thick] ([yshift=1pt]Z.north) -- ([yshift=-1pt]Z4.south) node[midway, right=0.1cm] {$\sigma$};
    \draw[->, thick] ([yshift=1pt]R.north) -- ([yshift=-1pt]Z.south) node[midway, right=0.1cm] {$\text{mod}_{2\pi}$};
    \draw[->, thick] ([yshift=1pt]I.north) -- ([yshift=-1pt]Wf.south) node[midway, left=0.1cm] {$A$};
    \draw[->, thick, bend right=35] ([xshift=2pt]R.east) to ([yshift=-2pt]ZZZ.south);
    \draw[->, thick, bend left=35] ([yshift=3pt]Wf.north) to ([xshift=-2pt]ZZZ.west);

    \node at (2.15,1.6) {\LARGE$\circlearrowright$};
    \node at (3.1,-1.2) {\LARGE$\circlearrowright$};
    \node at (3.7,4.7) {\LARGE$\circlearrowright$};
    \node at (6.3,1.5) {\LARGE$\circlearrowright$};
    \node at (3.2,6.4) {\large$\Theta_F$};
    \node at (8.3,0.4) {\large$\widetilde{\sigma}$};
\end{tikzpicture}
\end{center}

\mycomment{-0.4cm}
From the diagram above, we can deduce that any lift $\Theta_F$ satisfies the following equality
$$ \text{mod}_4\sspc \circ \Theta_F \circ A \spc =  \spc \text{mod}_4\sspc \circ \widetilde{\sigma} \circ \widetilde{s}\sspc. $$
This means that, there exists a function $\spc Q:[0,1]\longrightarrow 4\Z\spc$ such that
$  \Theta_F \circ A + Q=\widetilde{\sigma}\circ \widetilde{s}$. Since $\widetilde{\sigma} \circ \widetilde{s}$ and $\Theta_F \circ A$ are both continuous (with respect to the Khalimsky topology on $\Z$), the function $Q$ must be continuous as well. However, note that $4\Z \subset \Z$ is a discrete subset in the Khalimsky topology. This implies that $Q$ must be constant, and therefore, we can write
$$ \Theta_F \circ A + 4M  = \widetilde{\sigma}\circ \widetilde{s} $$
for some integer $M\in \Z$. Therefore, we can choose a suitable lift $\Theta_F$ that satisfies
$$ \Theta_F \circ A = \widetilde{\sigma}\circ \widetilde{s}.$$

At last, recall that $\widetilde{s}(0)$ and $\widetilde{s}(1)$ are both multiples of $\pi$, and therefore, we have that
\begin{equation*}
\widetilde{\sigma}\circ \widetilde{s}\sspc(0) = 2\spc\widetilde{s}\sspc(0)/\pi \quad \text{ and } \quad
\widetilde{\sigma}\circ \widetilde{s}\sspc(1) =2\spc \widetilde{s}\sspc(1)/\pi.
\end{equation*}
Let $x_1 \in \Z_1$ and $x_2 \in \Z_2$ be the points satisfying $A(0) = (x_1, F(x_1))$ and $A(1) = (x_2, F(x_2))$.
This means that
$ \Theta_F(x_1,F(x_1))= 2\spc \widetilde{s}\sspc(0)/\pi$ and $\Theta_F(x_2,F(x_2))= 2\spc \widetilde{s}\sspc(1)/\pi$, which implies
\mycomment{0.1cm}
$$ \text{Ind}(f,\O_1,\O_2) = \frac{1}{2\pi}\Bigl(\sspc\widetilde{s}\spc(1) - \widetilde{s}\spc(0)\sspc \Bigr) = \frac{1}{4}\Bigl(\sspc\Theta_F(x_2,F(x_2)) - \Theta_F(x_1,F(x_1))\sspc\Bigr).$$

\noindent Observe that the difference above is independent of the choice of lift $\Theta_F$, since they differ by an integer constant multiple of $4$.
This concludes the proof.
\end{proof}

\section{The index of a foliation between two transverse topological lines}\label{chap:foliation_index}

\mycomment{-0.4cm}

Let $\F$ be an oriented topological foliation of the plane. As shown in Lemma \ref{lem:parametrization}, the set
\mycomment{-0.4cm}
$$ W_\F = \Tp(\hspace*{0.08cm} \cdot \hspace*{0.08cm}, \hspace*{0.08cm}\cdot \hspace*{0.08cm}, \F )^{-1}(\sspc\dot 0\sspc)$$

\mycomment{-0.3cm}
\noindent is homeomorphic to $\R^2 \times (0,\infty)$, and therefore, it is simply connected. Consequently, the lifting theorem guarantees that $\spc\dot\theta(\hspace*{0.08cm} \cdot \hspace*{0.08cm}, \hspace*{0.08cm}\cdot \hspace*{0.08cm}, \mathcal H )\vert_{W_\F}$ admits a lift $\sspc\Theta_\F:W_\F\longrightarrow \Z\spc$, which is unique up to the addition of an integer multiple of $4$. This is represented in the diagram below.

\mycomment{-0.2cm}

\begin{center}
\hspace*{0.6cm}\begin{tikzpicture}[scale=0.75, node distance=2.2cm]
    \node (Wf) at (0,0) {$W_\F$};
    \node (Z4) at (4.3,0) {$\ZZ$};
    \node (Z) at (4.3,2.8) {$\mathbb{Z}$};
    
    \draw[->, thick] ([xshift=1pt]Wf.east) -- ([xshift=-1.5pt]Z4.west) node[midway, below=0.1cm] {$\spc\dot\theta(\hspace*{0.08cm} \cdot \hspace*{0.08cm}, \hspace*{0.08cm}\cdot \hspace*{0.08cm}, \H )$};
    \draw[->, thick] ([yshift=1pt]Wf.north east) -- ([yshift=-1pt]Z.south west) node[midway, above left=0.05cm] {$\Theta_\F$};
    \draw[->, thick] ([yshift=-2.5pt]Z.south) -- ([yshift=2.5pt]Z4.north) node[midway, right=0.1cm] {$\text{mod}_4$};
    
    \node at (2.85,0.9) {\LARGE$\circlearrowright$};
\end{tikzpicture}
\end{center}

\mycomment{-0.5cm}
Let $h:\R^2\longrightarrow \R^2$ be an orientation-preserving homeomorphism. We denote by $h\F$ the foliation obtained by applying $h$ to the leaves of $\F$, and by $\widehat{h\sspc}:\R^2\times \R^2\longrightarrow \R^2\times \R^2$ the homeomorphism $\widehat{h\sspc}(z,z') = (h(z), h(z'))$. First, observe that the following sets coincide
\mycomment{-0.1cm}
$$W_{h\F} = \{(h(z), h(z')) \mid z'\in \phi^+_z, z \neq z'\} = \widehat{h\sspc}(W_\F).$$

\mycomment{-0.2cm}
\noindent Moreover, note that any lift $\Theta_{h\F}$ respects the following commutative diagram.

\mycomment{-0.1cm}

\begin{center}
\hspace*{0.6cm}\begin{tikzpicture}[scale=0.85
    , node distance=2.2cm]
    \node (Wf) at (0,3) {$W_{h\F}$};
    \node (I) at (0,0) {$W_\F$};
    \node (Z4) at (4.7,3) {$\Z$};
    \node (Z) at (4.7,0) {$\ZZ$};
    
    \draw[->, thick] ([xshift=1pt]Wf.south east) -- ([xshift=-1.5pt]Z.north west) node[midway, sloped, above=0.1cm] {$\ \spc\dot\theta(\hspace*{0.08cm} \cdot \hspace*{0.08cm}, \hspace*{0.08cm}\cdot \hspace*{0.08cm}, \H )$};
    \draw[->, thick] ([xshift=1pt]I.east) -- ([xshift=-1.5pt]Z.west) node[midway, below=0.1cm] {$\spc\dot\theta(h(\hspace*{0.08cm} \cdot \hspace*{0.08cm}), \sspc h(\hspace*{0.08cm}\cdot \hspace*{0.08cm}), \H \sspc)$};
    \draw[->, thick] ([xshift=1pt]Wf.east) -- ([xshift=-1.5pt]Z4.west) node[midway, above=0.1cm] {$\Theta_{h\F}$};
    \draw[->, thick] ([yshift=-1pt]Z4.south) --([yshift=1pt]Z.north) node[midway, right=0.1cm] {$\text{mod}_4$};
    \draw[->, thick] ([yshift=1pt]I.north) -- ([yshift=-1pt]Wf.south) node[midway, left=0.1cm] {$\widehat{h}$};
    
    \node at (1.4,1) {\LARGE$\circlearrowright$};
    \node at (3.7,2.2) {\LARGE$\circlearrowright$};
\end{tikzpicture}
\end{center}

\mycomment{-0.45cm}
\noindent The diagram above shows that the map $\Theta_{h\F} \circ \widehat{h}$ is a lift of the function $\spc\dot\theta(h(\hspace*{0.08cm} \cdot \hspace*{0.08cm}), h(\hspace*{0.08cm}\cdot \hspace*{0.08cm}), \H )\vert_{W_\F}$.

Let $\Gamma_1$ and $\Gamma_2$ be two disjoint oriented topological lines on $\R^2$ that are positively transverse to the foliation $\F$. We associate to the ordered pair $(\Gamma_1,\sspc\Gamma_2)$ the following set
$$ \str(\Gamma_1,\Gamma_2) := \{\sspc h \in \homeo^+(\R^2) \mid h(\Gamma_1) = \R_1\sspc , \spc  h(\Gamma_2) = \R_2\sspc\},$$
where $\R_i$ denotes the horizontal line $\R\times\{i\}$, for $i \in \{1,2\}$. Note that, the set $\str(\Gamma_1,\Gamma_2)$ is non-empty (by the Homma-Schoenflies theorem), and it is path-connected, in the sense that, for any $h, h' \in \str(\Gamma_1,\Gamma_2)$, there exists an isotopy $(h_t)_{t\in [0,1]}$ such that $h_0 = h\sspc$, $\sspc h_1 = h'$ and $$h_t \in \str(\Gamma_1,\Gamma_2)\sspc, \quad \forall t\in [0,1].$$


\begin{proposition}\label{prop:index_invariance}
    For any $h \in \str(\Gamma_1,\Gamma_2)$, any lift $\Theta_{h\F}$, and any two points $x_i,x_i'\in \R^2$ satisfying $x_i \in \Gamma_i$ and $x_i' \in \phi^+_{x_i}\setminus \{x_i\}$, for all $i \in \{1,2\}$, the value
    $$ \Theta_{h\F}(\sspc h(x_2),\spc h(x_2')\sspc ) - \Theta_{h\F}(\sspc h(x_1),\spc h(x_1')\sspc )$$  
    is independent of all such choices.
\end{proposition}

\mycomment{-0.25cm}
\begin{proof}
    The independence of choice of lift $\Theta_{h\F}$ is clear, as it follows from the fact that any two lifts differ by an integer multiple of $4$. Thus, we begin by proving the independence of choice of map $h \in \str(\Gamma_1,\Gamma_2)$.
    Let $h_0, h_1 \in \str(\Gamma_1,\Gamma_2)$, and let $(h_t)_{t\in [0,1]}$ be an isotopy between the maps $h_0$ and $h_1$ that satisfies $h_t \in \str(\Gamma_1,\Gamma_2)$ for all $t\in [0,1]$. This isotopy induces a continuous map $\spc H: W_\F\times [0,1] \longrightarrow \ZZ\spc $ defined by the expression
    $$ H(t,z,z') = \Tp(\sspc h_t(z),\spc h_t(z')\sspc , \H).$$
    Since $W_\F\times [0,1]$ is simply-connected, the lifting theorem guarantees that the function $H$ admits a lift $\widetilde{H}:W_\F\times [0,1]\longrightarrow \Z$, as shown in the following commutative diagram.
    \begin{center}
\hspace*{0.6cm}\begin{tikzpicture}[scale=0.75, node distance=2.2cm]
    \node (Wf) at (0,0) {$W_\F\times[0,1]$};
    \node (Z4) at (4.3,0) {$\ZZ$};
    \node (Z) at (4.3,2.8) {$\mathbb{Z}$};
    
    \draw[->, thick] ([xshift=1pt]Wf.east) -- ([xshift=-1.5pt]Z4.west) node[midway, below=0.1cm] {$H$};
    \draw[->, thick] ([yshift=1pt]Wf.north east) -- ([yshift=-1pt]Z.south west) node[midway, above left=0.05cm] {$\widetilde{H}$};
    \draw[->, thick] ([yshift=-2.5pt]Z.south) -- ([yshift=2.5pt]Z4.north) node[midway, right=0.1cm] {$\text{mod}_4$};
    
    \node at (2.85,0.9) {\LARGE$\circlearrowright$};
\end{tikzpicture}
\end{center}

\mycomment{-0.4cm}
\noindent Observe that, for any $t\in [0,1]$, the map $\widetilde{H}(\hspace*{0.08cm} \cdot \hspace*{0.08cm}, t)$ is a lift of the function $\spc\dot\theta(h_t(\hspace*{0.08cm} \cdot \hspace*{0.08cm}), h_t(\hspace*{0.08cm}\cdot \hspace*{0.08cm}), \H )\vert_{W_\F}$. In other words, there exists a suitable choice of lifts $\{\Theta_{h_t \F}\}_{t\in [0,1]}$ satisfying
$$ \widetilde{H}(\hspace*{0.08cm} \cdot \hspace*{0.08cm},\hspace*{0.08cm} \cdot \hspace*{0.08cm}, t) = \Theta_{h_t \F}\circ \widehat{h_t}\sspc, \quad \forall t\in [0,1].$$
In particular, for this choice of lifts, we have that for each $i \in \{1,2\}$ the value
$$ \Theta_{h_t \F}(\sspc h_t(x_i),\spc h_t(x_i')\sspc ) \in \Z$$
varies continuously with respect to $t\in [0,1]$.

Now, observe that, since $h_t \in \str(\Gamma_1,\Gamma_2)$, and since the lines $\Gamma_1$ and $\Gamma_2$ are positively transverse to the foliation $\F$, for any $i \in \{1,2\}$ and $t\in [0,1]$, we have that
\mycomment{-0.07cm}
$$ \Tp(\sspc h_t(x_i),\spc h_t(x_i')\sspc , \H) \in \{\sspc -\dot{1
}\sspc, \spc \dot{1}\sspc\}.$$

\mycomment{-0.2cm}
\noindent
Consequently, for any $i \in \{1,2\}$ and $t\in [0,1]$, the lift $\Theta_{h_t \F}$ satisfies
\mycomment{-0.07cm}
$$ \Theta_{h_t \F}(\sspc h_t(x_i),\spc h_t(x_i')\sspc ) \in \{2k+1 \mid k\in \Z\}.$$

\mycomment{-0.2cm}
\noindent
At last, since $\{2k+1 \mid k\in \Z\}$ is a discrete subset of $\Z$ with the Khalimsky topology, and since the value $\Theta_{h_t \F}(\sspc h_t(x_i),\spc h_t(x_i')\sspc )$ varies continuously with $t\in [0,1]$, we conclude that 
\mycomment{-0.07cm}
$$\Theta_{h_1 \F}(\sspc h_1(x_i),\spc h_1(x_i')\sspc ) = \Theta_{h_0 \F}(\sspc h_0(x_i),\spc h_0(x_i')\sspc )\sspc, \quad \forall i \in \{1,2\}.$$

\mycomment{-0.2cm}
\noindent
This proves the independence of choice of map $h \in \str(\Gamma_1,\Gamma_2)$.

Finally, we prove the independence of choice of points $x_1$ and $x_1'$. The case of $x_2$ and $x_2'$ is completely analogous. Consider points  $x_1, y_1 \in \Gamma_1$, and let $x_1'\in \phi^+_{x_1}\setminus \{x_1\}$ and $y_1' \in \phi^+_{y_1}\setminus \{y_1\}$. Let $\alpha:[0,1]\longrightarrow \R^2$ be an injective sub-path of $\Gamma_1$ connecting $x_1$ and $y_1$. Note that $\alpha$ admits a neighborhood $V\subset \R^2$ that trivializes the foliation $\F$ and contains the points $x_1'$ and $y_1'$. This trivializing neighborhood allows us to construct a path $\alpha':[0,1]\longrightarrow V$ that is disjoint from the line $\Gamma_1$, connects the point $x_1'$ to $y_1'$, and satisfies $\alpha'(t) \in \phi^+_{\alpha(t)}$ for every $t\in [0,1]$. 
Since $\alpha'$ is disjoint from $\Gamma_1$, we know that the path $A:t\longmapsto (\alpha(t), \alpha'(t))$ is contained in $W_\F$. Moreover, since $\alpha$ is contained in $\Gamma_1$ and $h(\Gamma_1) = \R_1$, we know that the value
\mycomment{-0.07cm}
$$\Tp(h(\alpha(t)), h(\alpha'(t)), \H)$$

\mycomment{-0.2cm}
\noindent
is constant and equal to $-\dot{1}$ or $\dot{1}$ along all $t\in [0,1]$. Again, since $\{2k+1 \mid k\in \Z\}$ is a discrete subset of $\Z$ with the Khalimsky topology, and since $\Theta_{h\F}(\sspc h(\alpha(t)),\spc h(\alpha'(t))\sspc )$ varies continuously with $t\sspc$, we conclude that 
$\Theta_{h\F}(\sspc h(\alpha(t)),\spc h(\alpha'(t))\sspc )$
is  constant along $t\in [0,1]$. By evaluating this value at $t=0$ and $t=1$, we conclude that 
\mycomment{-0.07cm}
$$ \Theta_{h\F}(\sspc h(x_1),\spc h(x_1')\sspc ) =  \Theta_{h\F}(\sspc h(y_1),\spc h(y_1')\sspc ).$$

\mycomment{-0.2cm}
\noindent
This proves the independence of choice of points $x_1$ and $x_1'$, and concludes the proof.
\end{proof}

\mycomment{0.1cm}
With Proposition \ref{prop:index_invariance} in hand, we can finally define the index $\textup{Ind}(\F,\Gamma_1,\Gamma_2)$.

\begin{definition}
    The index of the foliation $\F$ between $\Gamma_1$ and $\Gamma_2$ is defined as
    $$\textup{Ind}(\F,\Gamma_1,\Gamma_2) := \frac{1}{4}\Bigl(\sspc\Theta_{h\F}(h(x_2),\spc h(x_2')\sspc ) - \Theta_{h\F}(h(x_1),\spc h(x_1')\sspc )\Bigr),$$
for any $h \in \str(\Gamma_1,\Gamma_2)$ and $x_i,x_i'\in \R^2$ such that $x_i \in \Gamma_i$ and $x_i' \in \phi^+_{x_i}\setminus \{x_i\}$, for all $i \in \{1,2\}$.
\end{definition}

To consider all cases, when $\Gamma_1$ and $\Gamma_2$ are not disjoint, we define $\textup{Ind}(\F,\Gamma_1,\Gamma_2):=0$.

\subsection{Accessible points in the configuration space}
\label{sec:accessible}

We begin the section by defining and explaining the concept of accessible points.

A point $z' \in \R^2$ is said to be \emph{accessible} (with respect to the foliation $\F$) from $z\in \R^2$ if either $z' \in \phi_z$ or there exists a path $\gamma:[0,1]\longrightarrow \R^2$ positively transverse to the foliation $\F$ with endpoints $\gamma(0) = z$ and $\gamma(1) = z'$.
For each point $z\in \R^2$, we can define the set 
$$A_z(\F) = A_z := \{z' \in \R^2 \mid z' \text{ is accessible from } z\}.$$
The set $A_z$ is a saturated by leaves of $\F$ and homeomorphic to the closed half-plane $[0,\infty)\times\R$, with the leaf $\phi_z$ being mapped to the vertical line $\{0\}\times\R$ oriented downwards. Moreover, the boundary $\partial A_z$ forms a locally-finite collection of leaves, with $A_z$ lying on the left side of each leaf in $\partial A_z$. It is worth noting that the set $A_z$ is not necessarily trivially foliated by $\F$. In fact, from the point of view of the leaf space of $\F$, the set $\pi_\F(A_z)$ corresponds in general to a non-Hausdorff simply connected 1-manifold oriented outwards from the leaf $\phi_z$.

\mycomment{0.2cm}

\begin{figure}[h!]
    \center\begin{overpic}[width=12cm, height=3.7cm, tics=10]{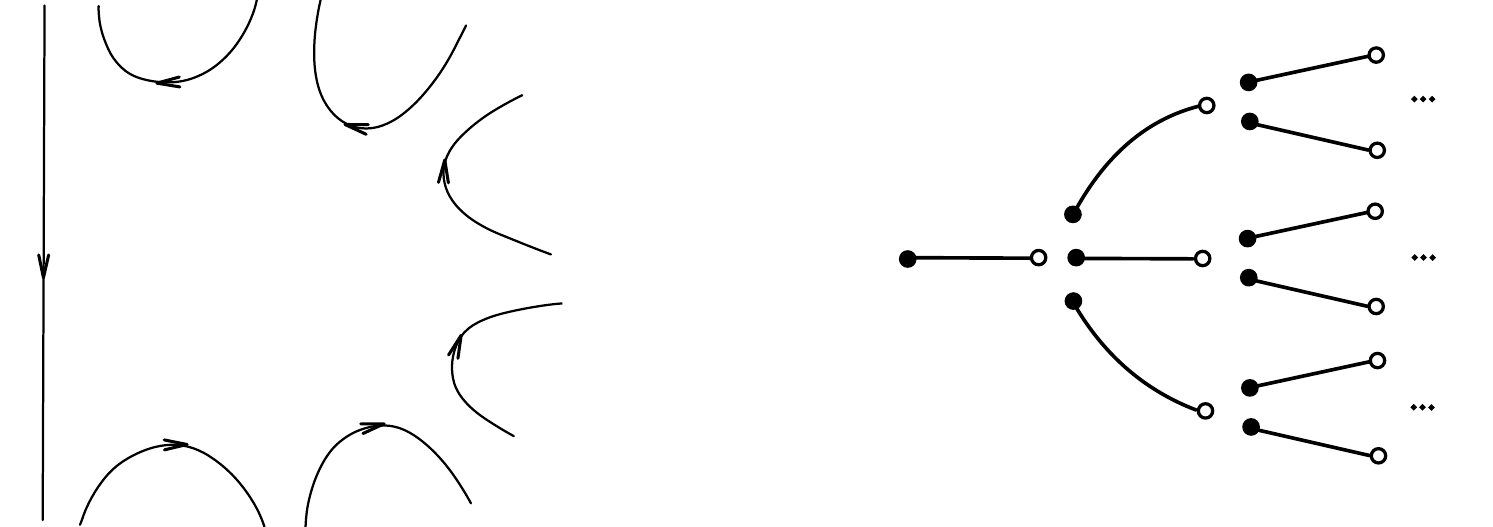}
         \put (15,14.5) {{\color{black}\large$\displaystyle A_z$}}
         \put (-3,14.5) {{\color{black}\large$\displaystyle \phi_z$}}
         \put (55,19) {{\color{black}\large$\displaystyle \pi_\F(\phi_z)$}}
         \put (74,-3) {{\color{black}\large$\displaystyle \pi_\F(A_z)$}}
\end{overpic}
\end{figure}

\mycomment{0.3cm}

Consider the subset $W_{A(\F)}$ of the configuration space $\R^2\times\R^2\setminus\Delta$ which is defined as
$$ W_{A(\F)} := \{(z,z') \in \R^2\times\R^2 \mid z' \in A_z(\F)\sspc, \spc z \neq z'\}.$$
Let $p:W_{A(\F)}\longrightarrow \R^2$ be the map given by projection on the first coordinate $p(z,z') = z$. 

\mycomment{0.3cm}

The rest of the section is dedicated to proving the following technical lemma, which will be crucial for the proof of Theorem \ref{thmx:III-A}.

\begin{lemma}
\label{lem:fibration}
    The map $p:W_{A(\F)}\longrightarrow \R^2$ is a Serre fibration.
\end{lemma}

\begin{proof}
    Let $\mathcal V$ be a countable cover of $\R^2$ consisting of open sets $V\in \mathcal V$ trivializing $\F$. More precisely, for each $V\in \mathcal V$, there exists a homeomorphism $\varphi: \overline{\rule{0cm}{0.3cm} V} \longrightarrow [0,1]^2$ such that, for each leaf $\phi\in \F$, the set $\varphi(\phi\cap V)$ is a vertical segment $\{s\}\times [0,1]$ oriented downwards.
    
    According to Theorem \ref{thm:local_serre_fibration}, to prove that $p:W_{A(\F)}\longrightarrow \R^2$ is a Serre fibration, it suffices to show that for each $V\in \mathcal V$, the restriction $p\vert_{p^{-1}(V)}:p^{-1}(V)\longrightarrow V$ is a Serre fibration.

    Therefore, we fix $V\in \mathcal V$, and a corresponding map  $\varphi:\overline{\rule{0cm}{0.3cm} V}\longrightarrow [0,1]^2$ as described above. 
    Let $\phi_R \in \F$ be the right-most leaf intersecting the set $\overline{\rule{0cm}{0.3cm} V}$, in other words $\phi_R := \phi_{\varphi^{-1}(0,0)}$, and let $\phi_L \in \F$ be the left-most leaf intersecting $\overline{\rule{0cm}{0.3cm} V}$, that is  $\phi_L = \phi_{\varphi^{-1}(1,0)}$. Observe that
    \mycomment{-0.1cm}
$$ A_{\varphi^{-1}(s',0)} \subset A_{\varphi^{-1}(s,0)}, \quad \forall 0\leq s < s' \leq 1.$$

    \mycomment{-0.2cm}
    \noindent This motivates us to denote $A:= A_{\varphi^{-1}(0,0)}$, as it satisfies $A_z \subset A$ for every point $z\in \overline{\rule{0cm}{0.3cm} V}$.

    Let $\Gamma:[0, \infty)\longrightarrow \R^2$ be a properly-embedded path, positively transverse to $\F$, that is contained in $A$ and satisfies $\Gamma(s) = \varphi^{-1}(\sspc s\sspc,\sspc \frac{1}{2}\sspc)$ for all $s\in [0,1]$. Define as well the set
    \mycomment{-0.1cm}
    $$ C:=\{ z \in \R^2 \mid \phi_z \cap \Gamma \neq \varnothing\}.$$

    \mycomment{-0.2cm}
    \noindent Note that $C$ is a subset of $A$ also saturated by leaves of $\F$ and homeomorphic to $[0,\infty)\times\R$, however, differently from $A$, the set $C$ is trivially foliated by $\F$. This means that, from the point of view of the leaf space of $\F$, the set $C$ corresponds to an oriented 1-manifold homeomorphic to the positively oriented interval $[0,\infty)$, with $\phi_R \in C$ being the point $0$.

    \mycomment{0.2cm}

\begin{figure}[h!]
    \center\begin{overpic}[width=12.5cm, height=4.1cm, tics=10]{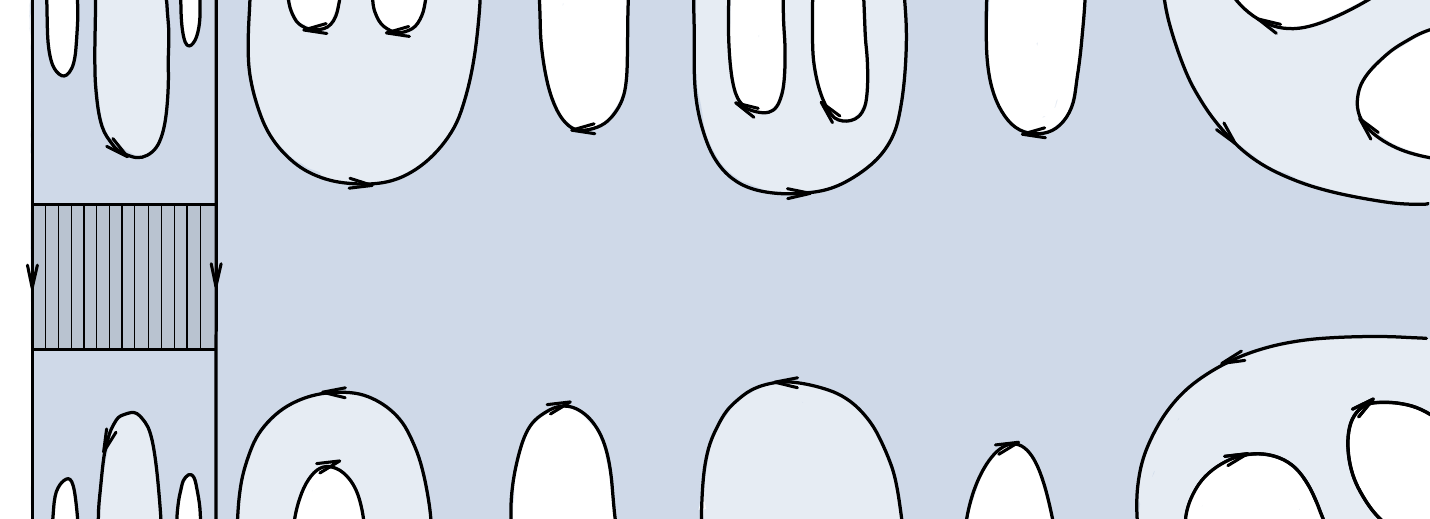}
         \put (55,13.5) {{\color{black}\large$\displaystyle C \subset A$}}
         \put (23,24.5) {{\color{black}\large$\displaystyle A$}}
         \put (3.2,21) {{\color{black}\large$\displaystyle \overline{\rule{0cm}{0.3cm} V}$}}
         \put (-3,14.5) {{\color{black}\large$\displaystyle \phi_R$}}
         \put (16.5,14.5) {{\color{black}\large$\displaystyle \phi_L$}}
\end{overpic}
\end{figure}

\mycomment{0cm}

    \noindent Now that we have defined the sets $A$ and $C$, we can finally begin the proof of the lemma. 

    \mycomment{0.2cm}

    Consider the $n$-dimensional cube $I^n=[0,1]^n$ and the following given objects:
    \begin{itemize}
        \item Let $H:[0,1]\times I^n \longrightarrow V$ be a homotopy, and denote $H_t := H\vert_{\{t\}\times I^n}$ for each $t\in [0,1]$.
        \item Let $\widetilde{H}_0:I^n \longrightarrow p^{-1} (V)$ be a lift of the map $H_0$ through the map $p\vert_{p^{-1}(V)}$.
    \end{itemize}
    \mycomment{0.2cm}
    Note that this implies that, for each $x\in I^n$, we can write
    \mycomment{-0.1cm}
        $$ \widetilde {H}_0(x) = (\sspc H_0(x)\sspc ,\sspc  H'_0(x)\sspc), \quad \forall x\in I^n,$$

    \mycomment{-0.2cm}
    \noindent where $H'_0:I^n\longrightarrow A$ is a continuous map satisfying $H'_0(x) \in A_{H_0(x)}\setminus \{H_0(x)\}$, for any $x\in I^n$.
    
    In order to prove that $p\vert_{p^{-1}(V)}:p^{-1}(V)\longrightarrow V$ is a Serre fibration, we need to construct a continuous family of continuous maps $H_t':I^n\longrightarrow A$ satisfying the accessibility condition
    $$ H'_t(x) \in A_{H_t(x)}\setminus \{H_t(x)\}, \quad \forall x\in I^n, \quad \forall t\in [0,1].$$
    This is because, the homotopy $\widetilde{H}:I^n\times [0,1]\longrightarrow p^{-1}(V)$ defined by
    $ \widetilde{H}(x,t) = (\sspc H_t(x)\sspc ,\sspc H'_t(x)\sspc)$
    will constitute a lift of the homotopy $H:I^n\times [0,1]\longrightarrow V$ satisfying $\widetilde{H}\vert_{\{0\}\times I^n} = \widetilde{H}_0$.

    \mycomment{-1cm}

    Before constructing the continuous family of maps $\{H_t'\}_{t\in [0,1]}$ in the general case, we first explore a simplified case under the assumption that $H_0'(I^n)$ is contained in $C$. Observe that, in this case, since $H_0'(I^n)$ is a compact subset of $C$, and since $C$ is trivially foliated by $\F$, we can consider another open set $U\subset \R^2$ trivializing the foliation $\F$ that satisfies
    \mycomment{-0.1cm}
    $$ U \subset C \quad \text{ and } \quad \overline{\rule{0cm}{0.3cm} V} \cup H_0'(I^n) \subset U.$$

    \mycomment{-0.25cm}
    \noindent Now, consider an isotopy $(g_t)_{t\in [0,1]}$ on $\R^2$, supported on $\overline{\rule{0cm}{0.3cm} U}$, that starts at the identity and moves the points positively transverse to $\F$, until the set $g_1(H_0'(I^n))$ is contained in $L(\phi_L) \cap U$. The speed at which $(g_t)_{t\in [0,1]}$ moves the points depends on the given isotopy $(H_t)_{t\in [0,1]}$, but by the compactness of $H_0'(I^n)$, we can ensure that the isotopy $(g_t)_{t\in [0,1]}$ can be taken to move the points sufficiently fast so that the accessibility condition 
    \mycomment{-0.12cm}
    $$ g_t(H_0'(x)) \in A_{H_t(x)}\setminus \{H_t(x)\} $$

    \mycomment{-0.2cm}
    \noindent remains satisfied for all $x\in I^n$ and $t\in [0,1]$. We remark that the condition $g_t(H_0'(x)) \neq H_t(x)$ is ensured by the fact that $(g_t)_{t\in [0,1]}$ moves points positively transverse to $\F$. In conclusion, in this simplified case, we attain the desired family $\{H_t'\}_{t\in [0,1]}$ by defining $H_t' := g_t\circ H_0'$.

    In the general case, it is possible that the set $H_0'(I^n)$ intersects regions of $A\setminus C$. These regions may become inaccessible as the homotopy $H_t$ progresses. To address this, we need to construct a continuous family of maps $\{H_t'\}_{t\in [0,1]}$ that contracts $H_0'(I^n)$ into the set $C$, before the regions of $A\setminus C$ that it intersects become inaccessible. However, we cannot simply retract the set $H_0'(I^n)$ from the regions of $A\setminus C$ without paying attention to the accessibility condition of the points already within $C$. This is because, there may exist points $x \in I^n$ such that $H_0(x)$ and $H_0'(x)$ lie on the same leaf of $\F$, and therefore, if $H_t(x)$ moves immediately to the left of the leaf $\phi_{H_t(x)}$, then the point $H_t'(x)$ needs to immediately follow the point $H_t(x)$ leftwards to maintain the accessibility condition. Once these details are taken into account, and the set $H_t'(I^n)$ is contained in $C$, we fall into the simplified case described above.

    Therefore, our method to construct the family $\{H_t'\}_{t\in [0,1]}$ follows three steps:
    \begin{itemize}[leftmargin=1.3cm]
        \item[(1)] To show that there exists a timeframe $t\in[0,T]$ for which the retraction of the set $H_0'(I^n)$ into $C$ can be performed while maintaining the accessibility condition.
        \item[(2)] To identify points $x\in I^n$ such that $H'_0(x)$ lose accessibility from $H_t(x)$ for $t\in [0,T]$, and to initially construct $\{H_t'\}_{t\in [0,T/2]}$ to ensure the accessibility from $H_t(x)$ to $H'_t(x)$.
        \item[(3)] To construct the family $\{H_t'\}_{t\in [T/2,T]}$ that retracts the set $H'_t(I^n)$ into $C$.
    \end{itemize}
    \mycomment{0.22cm}
    Once we have constructed the family $\{H_t'\}_{t\in [0,T]}$, the set $H'_T(I^n)$ will be contained in $C$, which means to construct the remaining family $\{H_t'\}_{t\in [T,1]}$ as in the simplified case.

    \vspace*{0.3cm}

    \textbf{Step 1 - Finding a timeframe for the retraction.}

    Let $\{B_i\}_{i\in J}$ be the collection of connected components of $A\setminus C$ that intersect $H_0'(I^n)$.
    The set of indices $J$ is finite, because $H_0'(I^n)$ is a compact set and $\partial C$ forms a locally-finite collection of leaves of $\F$ (see \cite[Lemma 2.4]{schuback1}).
    For each $i\in J$, let $\phi(B_i) \in \partial C$ be the leaf satisfying
    \mycomment{-0.1cm}
    $$ B_i = A \cap \overline{\rule{0cm}{0.32cm} L(\phi(B_i))}\quad \text{ and } \quad C \subset R(\phi(B_i)).$$
    
    \mycomment{-0.25cm}
    \noindent
    Note that, for each $i \in J$, the set $B_i \cap H_0'(I^n)$ is a compact. This implies that each set
    \mycomment{-0.1cm}$$ (H_0')^{-1}(B_i) = \{x\in I^n \mid H_0'(x) \in B_i\}$$

    \mycomment{-0.25cm}
    \noindent
    is a compact subset of $I^n$ and, consequently, that $ H_0(H_0')^{-1}(B_i))$ is a compact subset of $V$.

    \begin{claim}\label{claim:timeframe}
        For each $i\in J$, there exists a leaf $\phi_i \in \F$ that intersects $V$ and satisfies
        \mycomment{-0.1cm}$$ H_0((H_0')^{-1}(B_i)) \subset R(\phi_i) \quad \text{ and } \quad B_i \subset L(\phi_i).$$
    \end{claim}

    \vspace*{-0.1cm}
    \begin{figure}[h!]
    \center\begin{overpic}[width=3.6cm, height=5cm, tics=10]{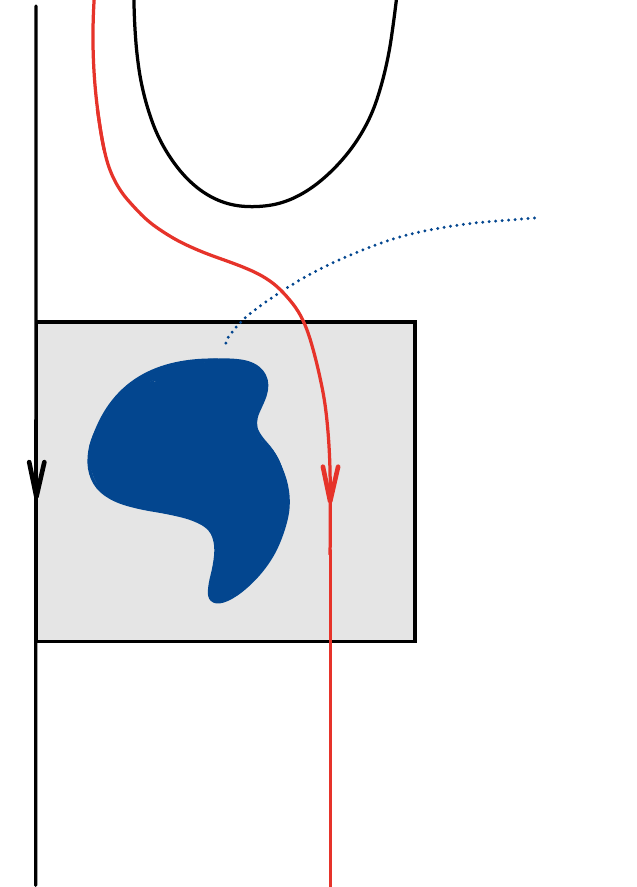}
         \put (7,30) {{\color{black}\large$\displaystyle V$}}
         \put (39,4) {{\color{myRED}\large$\displaystyle \phi_i$}}
         \put (24,88) {{\color{black}\large$\displaystyle B_i$}}
         \put (-8,5) {{\color{black}\large$\displaystyle \phi_R$}}
         \put (62,73) {{\color{myBLUE}\large$\displaystyle H_0((H_0')^{-1}(B_i))$}}
\end{overpic}
\end{figure}

\mycomment{-0.3cm}
    \begin{proof}[Proof of the Claim \ref{claim:timeframe}]
        Since $H_0((H_0')^{-1}(B_i))$ is compact subset of $V$, we can define the leaf $\phi_i^* \in \F$ as the left-most leaf intersecting $H_0((H_0')^{-1}(B_i))$. Let $q \in I^n$ be such that
        \mycomment{-0.1cm}$$ H_0(q) \in \phi_i^* \quad \text{ and } \quad H'_0(q) \in B_i.$$

    \mycomment{-0.25cm}
    \noindent
        Recall that $H'_0(q)$ is accessible from $H_0(q)$, however, they cannot be on the same leaf of $\F$, because $\phi_i^*$ is disjoint from $B_i$. Thus, there exists a path $\gamma:[0,1]\longrightarrow \R^2$ positively transverse to $\F$ such that $\gamma(0) = H_0(q)$ and $\gamma(1) = H'_0(q)$. Note that there exists $s_i\in [0,1]$ such that
        \mycomment{-0.1cm}$$ \gamma(s_i) \in \phi(B_i).$$

    \mycomment{-0.25cm}
    \noindent
        This means that, for any $s\in (\sspc 0,\sspc s_i\sspc )$, the leaf $\phi_{\gamma(s)}$ satisfies
        \mycomment{-0.1cm}$$ R(\phi_i^*) \subset R(\phi_{\gamma(s)}) \quad \text{ and } \quad B_i \subset L(\phi_{\gamma(s)}).$$

    \mycomment{-0.25cm}
    \noindent
    Moreover, there exists a sufficiently small $\varepsilon\in (\sspc 0,\sspc s_i\sspc )$ such that the leaf $\phi_{\gamma(\varepsilon)}$ intersects $V$. Finally, we can simply define the leaf $\phi_i$ to be the leaf $\phi_{\gamma(\varepsilon)}$. Since $\phi_i^*$ is the left-most leaf that intersects $H_0((H_0')^{-1}(B_i))$, we conclude that $\phi_i$ also satisfies
        $ H_0((H_0')^{-1}(B_i)) \subset R(\phi_i)$. This concludes the proof of the claim.
    \end{proof}

    \mycomment{-1cm}
    For each $i\in J$, the set $H_0((H_0')^{-1}(B_i))$ is a compact set contained on $R(\phi_i)$, and therefore, there exists some $T_i\in (0,1)$ such that $H_t((H_0')^{-1}(B_i))$ remains in $R(\phi_i)$ for all $t\in [0,T_i]$. We can finally determine our timeframe by defining
    $$ T: = \min_{i\in J} T_i.$$
    \noindent In the case where $J = \varnothing$, we can simply take $T = 1$. 

    \vspace{0.3cm}
    \textbf{Step 2 - Ensuring accessibility before the retraction.}

    Let $D\subset H_0'(I^n)$ be the subset defined as $D:= \bigcup_{t \in [0,T]} D_t$, where $D_t$ is defined as
    \mycomment{-0.1cm}$$ D_t := \{H_0'(x) \mid H_0'(x) \text{ is not accessible from } H_t(x)\sspc, \sspc x \in I^n\}.$$

    \mycomment{-0.25cm}
    \noindent
    Roughly speaking, the set $D$ contains all points in $H_0'(I^n)$ that lose accessibility from $H_t(x)$ for some $t\in [0,T]$. This is the set of points that we need to take care of in order to ensure the accessibility condition before we perform the retraction of $H_0'(I^n)$ into $C$.

    For $i \in J$, let $N_i\subset \R^2$ be a neighborhood of $B_i$ such that $N_i \subset L(\phi_i)$ and $N_i \cap \overline{\rule{0cm}{0.3cm} V} = \varnothing$. Observe that, from the definition of $T$, we have that $H_t((H_0')^{-1}(B_i)) \subset R(\phi_i)$ for all $t\in [0,T]$. Thus, by considering each neighborhood $N_i$ to be sufficiently small, we can also ensure that
    \mycomment{-0.1cm}$$ H_t((H_0')^{-1}(N_i)) \subset R(\phi_i), \quad \forall t\in [0,T].$$

    \mycomment{-0.2cm}
    \begin{claim}
    \label{claim:problem_set}
        There exists a compact set $K \subset C \setminus \bigcup_{i\in J} N_i$ that satisfies $D \subset K$.
    \end{claim}

    \begin{proof}[Proof of the Claim \ref{claim:problem_set}]
        Let $x \in (H'_0)^{-1}(N_i)$. Since the point $H_0'(x)$ is initially accessible from $H_0(x)$, and since $H_t(x)$ remains inside $V \cap R(\phi_i)$ for all $t\in [0,T]$, we can use the local trivial structure of the foliation $\F$ within $V\cap R(\phi_i)$ to show that the point $H_0'(x)$ remains accessible from $H_t(x)$ for all $t\in [0,T]$. In particular, $D$ is disjoint from each set $N_i$, implying
        $$ D \spc\subset\spc H_0'(I^n) \setminus \bigcup_{i\in J} N_i  \spc\subset\spc A \setminus \bigcup_{i\in J} N_i  \spc\subset\spc C \setminus \bigcup_{i\in J} N_i.$$
        We conclude the proof of the claim by defining the compact set $K := H_0'(I^n) \setminus \bigcup_{i\in J} N_i$.
    \end{proof}

    Using the fact that $C$ is trivially foliated by $\F$, we can consider an open set $U\subset \R^2$, contained in $C$, that locally trivializes the foliation $\F$ and satisfies $K\subset U$ and $\phi_L \cap U \neq \varnothing$.

    \mycomment{0.2cm}
    \begin{figure}[h!]
    \center\begin{overpic}[width=11cm, height=3.7cm, tics=10]{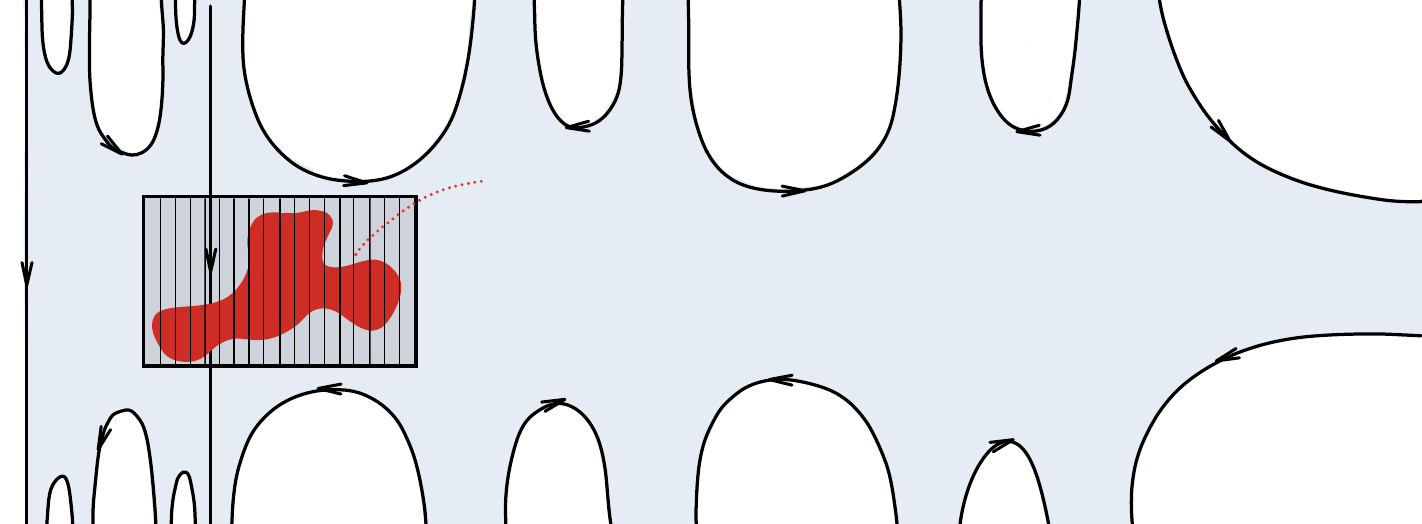}
         \put (0,-4) {{\color{black}\large$\displaystyle \phi_R$}}
        \put (14,-4) {{\color{black}\large$\displaystyle \phi_R$}}
        \put (6,14.5) {{\color{black}\large$\displaystyle U$}}
        \put (50,13.5) {{\color{black}\large$\displaystyle C$}}
         \put (35,21) {{\color{myRED}\large$\displaystyle K$}}
\end{overpic}
\end{figure}

    Similarly to the simplified case, consider an isotopy $(g_t)_{t\in [0,T/2]}$ on $\R^2$, supported on $\overline{\rule{0cm}{0.3cm} U}$, starting at the identity and moving  points positively transverse to $\F$, until the set $g_{T/2}(K)$ is contained in $L(\phi_L) \cap U$. As mentioned before, the speed at which $(g_t)_{t\in [0,T/2]}$ moves points depends on the given isotopy $(H_t)_{t\in [0,T/2]}$, but by the compactness of $K$, we can ensure that the isotopy $(g_t)_{t\in [0,T/2]}$ can be taken to move points sufficiently fast so that the accessibility condition
    $ g_t(H_0'(x)) \in A_{H_t(x)}\setminus \{H_t(x)\}$
    remains satisfied for all $x\in I^n$ and $t\in [0,T/2]$. 

    We can finally define the first fragment of the continuous family of maps $\{H_t'\}_{t\in [0,T/2]}$ by
    $$ H_t' := g_t\circ H_0', \quad \forall t\in [0,T/2].$$
    Observe that, points in $L(\phi_L) \cap U$ are automatically accessible from $V$. Since $K$ contains all points that may break the accessibility condition for times $t\in [0,T]$, we have that 
    $$ H'_{T/2}(x) \text{ is accessible from } H_t(x)\sspc, \quad \forall x\in I^n\sspc,\ t\in [T/2,T].$$

    With this property, we can finally perform the necessary retraction of $H'_{T/2}(I^n)$ into $C$.

    \vspace*{0.3cm}

    \textbf{Step 3 - Retraction isotopy.}

    Finally, for each $i \in J$, we can consider the isotopy $(h^{(i)}_t)_{t\in [T/2,T]}$ on $\R^2$ starting at the identity map, supported on the neighborhood $N_i$ that moves the set $H'_{T/2}(I^n)$ outside $B_i$, until $h^{(i)}_T(H'_{T/2}(I^n))$ lies entirely on $R(\phi(B_i))$. This isotopy is illustrated in the figure below.

    \mycomment{0.3cm}
    \begin{figure}[h!]
    \center\begin{overpic}[width=11cm, height=3.5cm, tics=10]{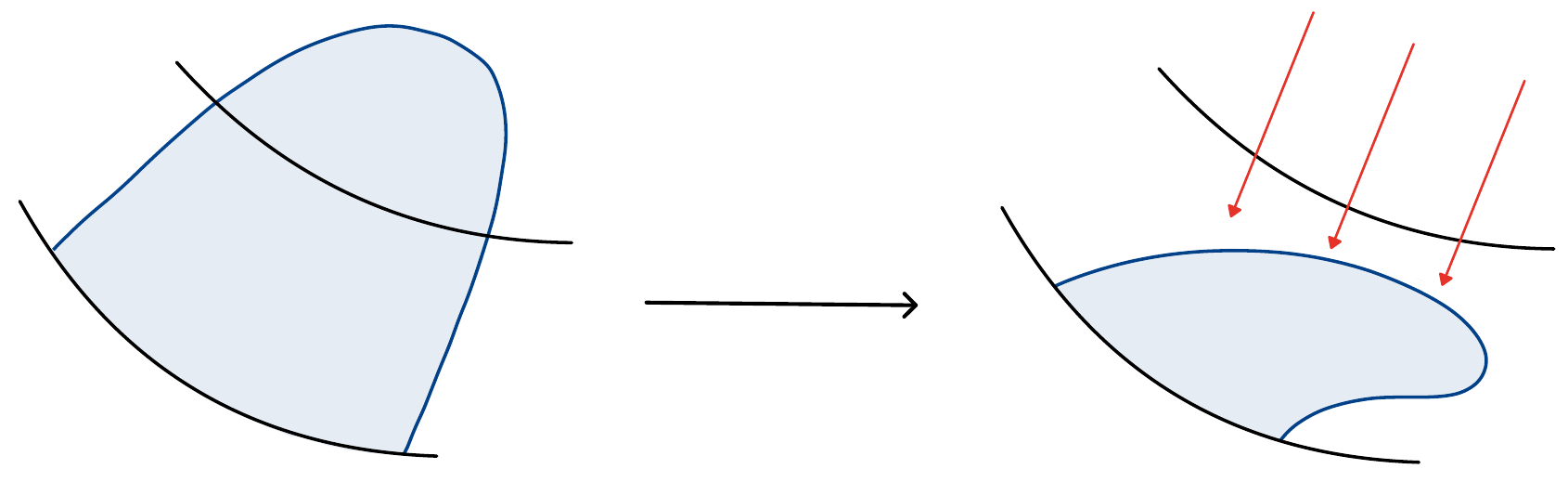}
        \put (-2,21) {{\color{black}\large$\displaystyle N_i$}}
        \put (4.5,30) {{\color{black}\large$\displaystyle \phi(B_i)$}}
        \put (60,21) {{\color{black}\large$\displaystyle N_i$}}
        \put (68,30) {{\color{black}\large$\displaystyle \phi(B_i)$}}
        \put (89.5,31.5) {{\color{myRED}\large$\displaystyle h_t^{(i)}$}}
        \put (41,13.5) {{\color{black} Retraction}}
        \put (44,8.3) {{\color{black} isotopy}}
         \put (16.99,23.5) {{\color{myBLUE}\large$\displaystyle H'_{T/2}(I^n)$}}
          \put (72.5,7.8) {{\color{myBLUE}\large$\displaystyle h_T(H'_{T/2}(I^n))$}}
\end{overpic}
\end{figure}

    \mycomment{-0.3cm}
    Since the neighborhoods in the finite family $\{N_i\}_{i \in J}$ are pairwise disjoint, we can define for each $t \in [T/2,T]$ the homeomorphism $h_t$ given by the composition of all maps $h^{(i)}_t$, $i \in J$. This composition yields an isotopy $(h_t)_{t\in [T/2,T]}$ on $\R^2$, that allows us to construct the second fragment of the continuous family of maps $\{H_t'\}_{t\in [T/2,T]}$, given by the expression
    $$ H_t' := h_t\circ H'_{T/2}, \quad \forall t\in [T/2,T].$$
    Observe that the family $\{H_t'\}_{t\in [T/2,T]}$ does not risk breaking the accessibility condition, since the set $H_t((H'_{T/2})^{-1}(N_i))$ remains in $R(\phi_i)$ for all $t\in [T/2,T]$. Now that the set $H_T'(I^n)$ is contained in $C$, we can finally construct the last fragment of the continuous family $\{H_t'\}_{t\in [T,1]}$ as in the simplified case. This concludes the proof of the lemma.
\end{proof}

\mycomment{-1.1cm}

\section{Proof of Theorem \ref{thmx:III-A}}

\mycomment{-0.4cm}
\begin{proof}[Proof of Theorem \ref{thmx:III-A}]
    Recall that two orbits of a Brouwer homeomorphism either cross a common leaf of the foliation $\F$, or are separated by a leaf of the foliation $\F$. 
    
    First, assume that $\O_1$ and $\O_2$ cross a common leaf of the foliation $\F$. As a consequence of Theorem \ref{thm:free}, we know there exists an isotopy $(f_t)_{t\in [0,1]}$ relative to $\O_1\cup \O_2$ such that each $f_t$ is a Brouwer homeomorphism, $f_0 = f$ and $f_1$ is conjugate to a translation. In particular, we have that $\textup{Ind}(f_1,\O_1,\O_2) = 0$. Recall that $\textup{Ind}(f,\O_1,\O_2)$ is invariant under isotopies of Brouwer homeomorphisms relative to the orbits $\O_1$ and $\O_2$. This implies that $\textup{Ind}(f_t,\O_1,\O_2)$ is constant along $t\in [0,1]$, and therefore, $\textup{Ind}(f,\O_1,\O_2) = 0$.

    Now, we show that in this case, the index $\textup{Ind}(\F,\Gamma_1,\Gamma_2)$ is also equal to $0$. Let $\phi\in \F$ be a leaf that is crossed by both $\O_1$ and $\O_2$. Note that, if $\Gamma_1$ and $\Gamma_2$ intersect each other, then we automatically have that $\textup{Ind}(\F,\Gamma_1,\Gamma_2) = 0$. Otherwise, we can apply the Homma-Schoenflies theorem to obtain an orientation-preserving homeomorphism $h:\R^2\longrightarrow \R^2$ satisfying
    \mycomment{-0.1cm}
$$h(\Gamma_1) = \R_1\sspc, \quad h(\Gamma_2) = \R_2\sspc, \quad \text{ and } \quad h(\phi) = \{0\} \times \R.$$

\mycomment{-0.2cm}
\noindent In particular, $h \in \str(\Gamma_1,\Gamma_2)$. Assume that $h$ maps the orientation of $\phi$ to the downward orientation of the vertical line $ h(\phi) = \{0\} \times \R$. It is important to note that this cannot be guaranteed in general, but the proof is completely analogous in the other case.

    Consider the paths $\alpha,\alpha':[0,1]\longrightarrow \R^2$ defined by $\alpha(t) = (0,t+1)$ and $\alpha'(t) = (0,t+1 - \varepsilon)$, for some $\varepsilon > 0$. Note that the path $t\longmapsto (\alpha(t), \alpha'(t))$ is contained in $W_{h\F}$, and satisfies
    $$ \Tp(\alpha(t), \alpha'(t), \H) = -\dot{1}\sspc, \quad \forall t\in [0,1].$$
 This means that $\Theta_{h\F}(\alpha(t), \alpha'(t))$ is constant along $t\in [0,1]$. In particular, by considering the two points $x_1 \in \Gamma_1$ and $x_2 \in \Gamma_2$ which satisfy $h(x_1) = \alpha(0)$ and $h(x_2) = \alpha(1)$, and the two points $x_1' \in \phi^+_{x_1}$ and $x_2' \in \phi^+_{x_2}$ that satisfy $h(x_1') = \alpha'(0)$ and $h(x_2') = \alpha'(1)$, we get that
    $$ \Theta_{h\F}(h(x_2), h(x_2')) = \Theta_{h\F}(h(x_1), h(x_1')).$$
    Therefore, we conclude that $\textup{Ind}(\F,\Gamma_1,\Gamma_2) = 0 = \textup{Ind}(f,\O_1,\O_2)$ in this first case.

    \mycomment{-0.65cm}
    Now, we consider the case where $\O_1$ and $\O_2$ are separated by a leaf of the foliation $\F$. In this case, $\Gamma_1$ and $\Gamma_2$ are automatically disjoint, and we can consider the set $\str^\dagger(\Gamma_1,\Gamma_2)$, formed by all $h \in \str(\Gamma_1,\Gamma_2)$ satisfying the additional condition $h(\O_1) = \Z_1$ and $h(\O_2) = \Z_2$.

    \begin{claim}\label{claim:str_dagger}
        $\str^\dagger(\Gamma_1,\Gamma_2)\subset \textup{Handel}(\Gamma_1,\Gamma_2)$.
    \end{claim}

    \begin{proof}[Proof of Claim \ref{claim:str_dagger}]
        As a consequence of \cite[Corollary 6.2]{schuback2}, we know that there exists an isotopy $(f_t)_{t\in [0,1]}$ relative to $\O_1\cup \O_2$ such that $f_0 = f$ and $f_1$ is the time-one map of a flow having $\Gamma_1$ and $\Gamma_2$ as trajectories.

        Let $h \in \str^\dagger(\Gamma_1,\Gamma_2)$. Observe that, by performing the Alexander trick on each connected component of the complement $\R^2\setminus (\Gamma_1\cup \Gamma_2)$, we can construct an isotopy $(F_t)_{t\in [0,1]}$ relative to $\Gamma_1\cup \Gamma_2$ such that $F_0 = h \circ f_1 \circ h^{-1}$, and $F_1$ is either $T$, $T^{-1}$, $R$, or $R^{-1}$. This means that the map $h \circ f_1 \circ h^{-1}$, and therefore $h \circ f \circ h^{-1}$ as well, is isotopic relative to $\O_1 \cup \O_2$ to either $T$, $T^{-1}$, $R$, or $R^{-1}$. In particular, this implies that $h \in \textup{Handel}(\Gamma_1,\Gamma_2)$.
    \end{proof}

    Consider the set $W_{A(\F)}$ of accessible points in $ \R^2\times \R^2\setminus \Delta$, as defined in Section \ref{sec:accessible}. As shown by Lemma \ref{lem:fibration}, the projection $p:W_{A(\F)}\longrightarrow \R^2$, $p(z,z') = z$, is a Serre fibration. Since the base space $\R^2$ and the fibers $p^{-1}(z) = A_z\setminus\{z\}$, for $z\in \R^2$, are simply-connected, the Corollary \ref{cor:simpli-connected-fibrations} tells us that $W_{A(\F)}$ is simply-connected. Therefore, the lifting theorem guarantees that the function
    $\spc\dot\theta(\hspace*{0.08cm} \cdot \hspace*{0.08cm}, \hspace*{0.08cm}\cdot \hspace*{0.08cm}, \H )\vert_{W_{A(\F)}}$ admits a lift $\Theta_{A(\F)}:W_{A(\F)}\longrightarrow \Z$, which is unique up to the addition of an integer multiple of $4$.
    
    This is represented in the commutative diagram below.

    \begin{center}
\hspace*{0.6cm}\begin{tikzpicture}[scale=0.85, node distance=2.2cm]
    \node (Wf) at (0,0) {$W_{A(\F)}$};
    \node (Z4) at (4.3,0) {$\ZZ$};
    \node (Z) at (4.3,2.8) {$\mathbb{Z}$};
    
    \draw[->, thick] ([xshift=1pt]Wf.east) -- ([xshift=-1.5pt]Z4.west) node[midway, below=0.1cm] {$\spc\dot\theta(\hspace*{0.08cm} \cdot \hspace*{0.08cm}, \hspace*{0.08cm}\cdot \hspace*{0.08cm}, \H )$};
    \draw[->, thick] ([yshift=1pt]Wf.north east) -- ([yshift=-1pt]Z.south west) node[midway, above left=0.05cm] {$\Theta_{A(\F)}$};
    \draw[->, thick] ([yshift=-2.5pt]Z.south) -- ([yshift=2.5pt]Z4.north) node[midway, right=0.1cm] {$\text{mod}_4$};
    
    \node at (2.85,0.9) {\LARGE$\circlearrowright$};
\end{tikzpicture}
\end{center}

\mycomment{-0.3cm}
Next, observe that both $W_\F$ and $W_f$ are subsets of $W_{A(\F)}$. The inclusion $W_\F\subset W_{A(\F)}$ is trivial from their definitions. On the other hand, the inclusion $W_f\subset W_{A(\F)}$ follows from the fact that $\F$ is a transverse foliation of $f$, as it means that $f(z) \in L(\phi_z)$ for every $z\in \R^2$. As a consequence of these inclusions, we can restrict the lift $\Theta_{A(\F)}$ to the sets $W_\F$ and $W_f$, obtaining lifts of the maps $\Tp(\hspace*{0.08cm} \cdot \hspace*{0.08cm}, \hspace*{0.08cm}\cdot \hspace*{0.08cm}, \F )\vert_{W_\F}$ and $\Tp(\hspace*{0.08cm} \cdot \hspace*{0.08cm}, \hspace*{0.08cm}\cdot \hspace*{0.08cm}, f )\vert_{W_f}$, respectively. In other words,
$$ \Theta_{A(\F)}\vert_{W_\F} = \Theta_\F \quad \text{ and } \quad \Theta_{A(\F)}\vert_{W_f} = \Theta_f.$$

\mycomment{-1cm}
   \noindent This means that, using Claim \ref{claim:str_dagger}, we can compute both indices as
   \begin{align*}
    \textup{Ind}(f,\O_1,\O_2) &= \frac{1}{4}\Bigl(\sspc\Theta_{A(h\F)}(h(x_2), h\circ f(x_2)) - \Theta_{A(h\F)}(h(x_1), h\circ f(x_1))\sspc \Bigr),\\[0.4ex]
    \textup{Ind}(\F,\Gamma_1,\Gamma_2) &= \frac{1}{4}\Bigl(\sspc\Theta_{A(h\F)}(h(x_2), h(x_2')) - \Theta_{A(h\F)}(h(x_1), h(x_1'))\sspc \Bigr),
   \end{align*}
   for any map $h \in \str^\dagger(\Gamma_1,\Gamma_2)$, and any two points $x_i \in \O_i$ and $x_i' \in \phi^+_{x_i}\setminus\{x_i\}$, for $i=\{1,2\}$.

   Therefore, for each $i \in \{1,2\}$, fix a point $x_i \in \O_i$, and let $\gamma_i:[0,1]\longrightarrow \R^2$ be a sub-path of the transverse trajectory $\Gamma_i$ that connects $x_i$ to $f(x_i)$. Let $V_1$ and $V_2$ be disjoint neighborhoods of $\gamma_1$ and $\gamma_2$, respectively, that locally trivialize the foliation $\F$. For each $i\in \{1,2\}$, we fix a point $x_i' \in \phi^+_{x_i}\setminus\{x_i\}$ that is contained in the neighborhood $V_i$.

   By the Homma-Schoenflies theorem, there exists a $h \in \str^\dagger(\Gamma_1,\Gamma_2)$ that satisfies
   \mycomment{0.1cm}
   $$ h(\sspc \overline{\rule{0cm}{0.3cm} V_i}\sspc ) = \left[-\frac{1}{2},\frac{3}{2}\right]\times \left[i-\frac{1}{4}, i +\frac{1}{4}\right], \quad \forall i\in \{1,2\}.$$
   This allows us to consider, for each $i\in \{1,2\}$, a line segment $\alpha'_i:[0,1]\longrightarrow \R^2$ that connects $h(x_i')$ to $h\circ f(x_i')$, which is automatically positively transverse to $h\F$ and disjoint from $x_i$.
   This means that, if we define $\alpha_i:[0,1]\longrightarrow \R^2$ as the constant path at $h(x_i)$, then the path 
   \mycomment{-0.1cm}
   $$ A_i: t \longmapsto (\alpha_i(t), \alpha'_i(t))$$

    \mycomment{-0.2cm}
    \noindent is contained in $W_{A(h\F)}$ and connects $(h(x_i), h(x_i'))$ to $(h(x_i), h\circ f(x_i))$, for each $i\in \{1,2\}$.

   Observe that, the construction of the paths $A_i$ guarantees that
   \mycomment{-0.1cm}
   $$\Theta_{A(h\F)} \circ A_i (t) = \Theta_{A(h\F)}(h(x_i), h(x_i')), \quad \forall t\in [0,1),$$

    \mycomment{-0.2cm}
    \noindent As the set $A_i([0,1))$ is entirely contained below, or above, the horizontal $h(\Gamma_i) =\R_i$, depending on the initial position of the point $h(x_i')$ relative to $\R_i$. Now, observe that, when the point $h(x_i')$ is located below the horizontal line $\R_i$, the point $h\circ f(x_i')$ must lie on the horizontal $\R_i$ with a first coordinate greater than the first coordinate of $h(x_i')$. Meanwhile, when the point $h(x_i')$ is located above $\R_i$, the point $h\circ f(x_i')$ must lie on the horizontal $\R_i$ with a first coordinate less than the first coordinate of $h(x_i')$. In both cases, we have that the value of $\Theta_{A(h\F)} \circ A_i (1)$ increases by $1$ with respect to the value of $\Theta_{A(h\F)} \circ A_i (t)$ for $t \in [0,1)$. 
   This allows us to conclude that, for each $i\in \{1,2\}$, we have the following equality
   $$ \Theta_{A(h\F)}(h(x_i), h\circ f(x_i)) = \Theta_{A(h\F)}(h(x_i), h(x_i')) + 1.$$
   By substituting this equality into the expression for the index $\textup{Ind}(f,\O_1,\O_2)$, as given above, this difference of $1$ cancels out, and we finally obtain the equality
   \mycomment{-0.1cm}
   $$ \textup{Ind}(f,\O_1,\O_2) = \textup{Ind}(\F,\Gamma_1,\Gamma_2).$$
   This concludes the proof of Theorem \ref{thmx:III-A}.
\end{proof}

\setstretch{1.6}	

    \bibliography{biblio1.bib}

\end{document}